\documentclass[11pt]{article}
\usepackage{amsmath,amsfonts,amssymb,latexsym,amsthm,dsfont,setspace}
\usepackage{geometry,graphicx}
\usepackage[backref]{hyperref}
\usepackage{xcolor}
\title{On fine properties of mixtures with respect to\\
  concentration of measure and Sobolev type inequalities}

\author{Djalil~\textsc{Chafa\"\i} and Florent~\textsc{Malrieu}}

\date{Preprint 2009}

\newtheorem{thm}{Theorem}[section]%
\newtheorem{cor}[thm]{Corollary}%
\newtheorem{lem}[thm]{Lemma}%
\newtheorem{xpl}[thm]{Example}%
\newtheorem{rem}[thm]{Remark}%

\newcommand{\dE}{\mathbb{E}}

\newcommand{\dR}{\mathbb{R}}

\newcommand{\cB}{\mathcal{B}}
\newcommand{\cC}{\mathcal{C}}

\newcommand{\cL}{\mathcal{L}}
\newcommand{\cM}{\mathcal{M}}\newcommand{\cN}{\mathcal{N}}

\newcommand{\cS}{\mathcal{S}}
\newcommand{\cU}{\mathcal{U}}
\newcommand{\cX}{\mathcal{X}}

\newcommand{\bE}{\mathbf{E}}

\newcommand{\al}{\alpha}
\newcommand{\be}{\beta}

\newcommand{\de}{\delta}
\newcommand{\ga}{\gamma}

\newcommand{\la}{\lambda}

\newcommand{\na}{\nabla}

\newcommand{\si}{\sigma}
\newcommand{\Te}{\Theta}
\newcommand{\te}{\theta}

\newcommand{\vphi}{\varphi}

\newcommand{\p}[4]{{#3}\!\left#1{#4}\right#2} 

\newcommand{\PAR}[1]{{{\left(#1\right)}}} % (1)
\newcommand{\ABS}[1]{{{\left| #1 \right|}}} % |1|
\newcommand{\BRA}[1]{{{\left\{#1\right\}}}} % {1}
\newcommand{\DP}[1]{{{\left<#1\right>}}} % <1>
\newcommand{\NRM}[1]{{{\left\| #1\right\|}}} % ||1||
\newcommand{\pd}{{\partial}} % derivee partielle
\newcommand{\LIP}[1]{{\|#1\|_{\mathrm{Lip}}}} % Lipschitz norm

\newcommand{\entf}[1]{\mathbf{Ent}_{#1}}
\newcommand{\ent}[2]{\p(){\entf{#1}}{#2}}
\newcommand{\moyf}[1]{\bE_{#1}}

\newcommand{\varf}[1]{\mathbf{Var}_{#1}}
\newcommand{\var}[2]{\p(){\varf{#1}}{#2}}

\newcommand{\OL}[1]{\overline{#1}}
\newcommand{\WH}[1]{\widehat{#1}}

\renewcommand{\leq}{\leqslant}
\renewcommand{\geq}{\geqslant}
\newcommand{\ds}[1]{\displaystyle{#1}}

\begin{document}

\maketitle

\begin{abstract}
  Mixtures are convex combinations of laws. Despite this simple definition, a
  mixture can be far more subtle than its mixed components. For instance,
  mixing Gaussian laws may produce a potential with multiple deep wells. We
  study in the present work fine properties of mixtures with respect to
  concentration of measure and Sobolev type functional inequalities. We
  provide sharp Laplace bounds for Lipschitz functions in the case of generic
  mixtures, involving a transportation cost diameter of the mixed family.
  Additionally, our analysis of Sobolev type inequalities for two-component
  mixtures reveals natural relations with some kind of band isoperimetry and
  support constrained interpolation via mass transportation. We show that the
  Poincar\'e constant of a two-component mixture may remain bounded as the
  mixture proportion goes to $0$ or $1$ while the logarithmic Sobolev constant
  may surprisingly blow up. This counter-intuitive result is not reducible to
  support disconnections, and appears as a reminiscence of the
  variance-entropy comparison on the two-point space. As far as mixtures are
  concerned, the logarithmic Sobolev inequality is less stable than the
  Poincar\'e inequality and the sub-Gaussian concentration for Lipschitz
  functions. We illustrate our results on a gallery of concrete two-component
  mixtures. This work leads to many open questions.
\end{abstract}

{\footnotesize %
\noindent\textbf{Keywords.} 
Mixtures of distributions; finite Gaussian mixtures; concentration of measure;
Gaussian bounds; tails probabilities; deviation inequalities; functional
inequalities; Poincar\'e inequalities; Gross logarithmic Sobolev inequalities;
band isoperimetry; transportation of measure; mass transportation;
transportation cost distances; Mallows or Wasserstein distance.

\medskip

\noindent\textbf{AMS-MSC.}  60E15; 49Q20; 46E35; 62E99 %
} % footnotesize

{
\footnotesize
\tableofcontents
}

\section{Introduction}

Mixtures of distributions are ubiquitous in Stochastic Analysis, Modelling,
Simulation, and Statistics, see for instance the monographs
\cite{MR624267,MR2265601,MR1789474,MR926484,MR838090}. Recall that a mixture
of distributions is nothing else but a \emph{convex combination} of these
distributions. For instance, if $\mu_0$ and $\mu_1$ are two laws on the same
space, and if $p\in[0,1]$ and $q=1-p$, then the law $p\mu_1+q\mu_0$ is a
``two-component mixture''. More generally, a \emph{finite mixture} takes the
form $p_1\mu_1+\cdots+p_n\mu_n$ where $\mu_1,\ldots,\mu_n$ are probability
measures on a common measurable space and $p_1\de_1+\cdots+p_n\de_n$ is a
finite discrete probability measure. A widely used example is given by finite
mixtures of Gaussians for which $\mu_i=\cN(m_i,\si_i^2)$ for every $1\leq
i\leq n$. In that case, for certain choices of $m_1,\ldots,m_n$ and
$\si_1,\ldots,\si_n$, the mixture
$$
p_1\cN(m_1,\si_1^2)+\cdots+p_n\cN(m_n,\si_n^2)
$$
is multi-modal and its log-density is a multiple wells potential. For
instance, each component $\mu_i$ may correspond typically in Statistics to a
sub-population, in Information Theory to a channel, and in Statistical Physics
to an equilibrium. Another very natural example is given by the invariant
measures of finite Markov chains, which are mixtures of the invariant measures
uniquely associated to each recurrent class of the chain. A more subtle
example is the local field of the Sherrington-Kirkpatrick model of spin
glasses which gives rise to a mixture of two univariate Gaussians with equal
variances, see for instance \cite{chatterjee-2007}.

At this point, it is enlightening to introduce a bit more abstract point of
view. Let $\nu$ be a probability measure on some measurable space $\Te$ and
$(\mu_\te)_{\te\in\Te}$ be a collection of probability measures on some common
fixed measurable space $\cX$, such that the map $\te\mapsto\bE_{\mu_\te}f$ is
measurable for any fixed bounded continuous $f:\cX\to\dR$. The mixture
$\cM(\nu,\mu_{\te\in\Te})$ is the law on $\cX$ defined for any bounded
measurable $f:\cX\to\dR$ by
$$
\moyf{\cM(\nu,\mu_{\te\in\Te})}{f} %
=\int_\Te\!\int_\cX\!f(x)\,d\mu_\te(x)\,d\nu(\te) %
=\bE_{\nu}(\te\mapsto\bE_{\mu_\te}f).
$$
Here $\nu$ is the \emph{mixing law} whereas $(\mu_\te)_{\te\in\Te}$ are the
\emph{mixed laws} or the \emph{mixture components} or even the \emph{mixed
  family}. With these new notations, and for the finite mixture example
mentioned earlier we have $\Te=\{1,\ldots,n\}$ and
$\nu=p_1\de_1+\cdots+p_n\de_n$ and
$$
\cM(\nu,(\mu_\te)_{\te\in\Te}) %
=\cM(p_1\de_1+\cdots+p_n\de_n,\{\mu_1,\ldots,\mu_n\}) %
=p_1\mu_1+\cdots+p_n\mu_n.
$$
The mixture $\cM(\nu,\mu_{\te\in\Te})$ can be seen as a sort of general convex
combination in the convex set of probability measures on $\cX$. It appears for
certain class of $\nu$ as a particular Choquet's Integral, see
\cite{MR1835574} and \cite{MR1863703}. On the other hand, the case where the
mixture components are product measures is also related to exchangeability and
De Finetti's Theorem, see for instance \cite{MR2137450}. In terms of random
variables, if $(X,Y)$ is a couple of random variables then the law $\cL(X)$ of
$X$ is a mixture of the family of conditional laws $\cL(X|Y=y)$ with the
mixing law $\cL(Y)$. By this way, mixing appears as the dual of the so-called
disintegration of measure. Here and in the whole sequel, the term ``mixing''
refers to the mixture of distributions as defined above and has \emph{a
  priori} nothing to do with weak dependence.

Our first aim is to investigate the fine behavior of concentration of measure
for mixtures, for instance for a two-component mixture $p\mu_1+q\mu_0$ as
$\min(p,q)$ goes to $0$. It is well known that Poincar\'e and (Gross)
logarithmic Sobolev functional inequalities are powerful tools in order to
obtain concentration of measure. Also, our second aim is to investigate the
fine behavior of these functional inequalities for mixtures, and in particular
for two-component mixtures. Our work reveals striking unexpected phenomena. In
particular, our results suggest that the logarithmic Sobolev inequality, which
implies sub-Gaussian concentration, is very sensitive to mixing, in contrast
with the sub-Gaussian concentration itself which is far more stable. As in
\cite{MR1814423} and \cite{MR2317340}, our work is connected to the more
general problem of the behavior of optimal constants for sequences of
probability measures.

Let us start with the notion of \emph{concentration of measure} for Lipschitz
functions. We denote by $\NRM{\cdot}_2$ the Euclidean norm of $\dR^d$. A
function $f:\dR^d\to\dR$ is Lipschitz when
$$
\LIP{f}=\sup_{x\neq y}\frac{|f(x)-f(y)|}{\NRM{x-y}_2}<\infty.
$$ 
Let $\mu$ be a law on $\dR^d$ such that $\bE_\mu|f|<\infty$ for every
Lipschitz function $f$. This holds true for instance when $\mu$ has a finite
first moment. We always make implicitly this assumption in the sequel. We
define now the log-Laplace transform $\al_\mu:\dR\to[0,\infty]$ of $\mu$ by
\begin{equation}\label{eq:def-alpha}
  \al_\mu(\la)= %
  \log \sup_{\LIP{f}\leq 1} %
  \bE_\mu\PAR{e^{\la(f-\moyf{\mu}{f})}}.
\end{equation}
The Cram\'er-Chernov-Chebychev inequality gives, for every $r>0$,
\begin{equation}\label{eq:def-beta}
  \be_\mu(r)= \sup_{\LIP{f}\leq 1} %
  \mu\PAR{\ABS{f-\bE_\mu f}\geq r} %
  \leq 2\exp\PAR{-\sup_{\la>0}\PAR{r\la-\al_\mu(\la)}}
\end{equation}
and the supremum in the right hand side is a Fenchel-Legendre transform of
$\al_\mu$. Note that $\be_\mu$ is a uniform upper bound on the tails
probabilities of Lipschitz images of $\mu$. We are interested in the control
of $\be_\mu$ via $\al_\mu$ in the case where
$\mu=\cM(\nu,(\mu_\te)_{\te\in\Te})$, in terms of the mixing law $\nu$ and of
the log-Laplace bounds $(\al_{\mu_\te})_{\te\in\Te}$ for the mixed family.

We say that $\mu$ satisfies a \emph{sub-Gaussian concentration of measure} for
Lipschitz functions when there exists a constant $C\in(0,\infty)$ such that
for every real number $\la$,
\begin{equation}\label{eq:subgau-al}
  \al_\mu(\la)\leq \frac{1}{4}C\la^2.
\end{equation}
The $\log$-Laplace-Lipschitz quadratic bound \eqref{eq:subgau-al} implies via
\eqref{eq:def-beta} that for every $r>0$,
\begin{equation}\label{eq:subgau-be}
  \be_\mu(r)\leq 2\exp\PAR{-\frac{r^2}{C}}.
\end{equation}
Actually, it was shown (see \cite{MR2078555} and \cite{bolley-villani}) that
up to constants, \eqref{eq:subgau-al} and \eqref{eq:subgau-be} are equivalent,
and are also equivalent to the existence of a constant
$\varsigma\in(0,\infty)$ and $x_0\in\dR^d$ such that
\begin{equation}\label{eq:int-expsq}
  \int_{\dR^d}\!e^{\varsigma|x-x_0|^2}\,\mu(dx)<\infty.
\end{equation}
Linear or quadratic upper bounds for $\al_\mu$ may be deduced from functional
inequalities such as Poincar\'e and (Gross) logarithmic Sobolev inequalities
\cite{MR0420249,MR2325763}. We say that $\mu$ satisfies a Poincar\'e
inequality of constant $C\in(0,\infty)$ when for every smooth $h:\dR^d\to\dR$,
\begin{equation}\label{eq:def-PI}
  \mathbf{Var}_\mu(h)\leq C\,\bE(|\nabla h|^2)
\end{equation}
where $\mathbf{Var}_\mu(h)=\bE_\mu(h^2)-(\bE_\mu h)^2$ is the \emph{variance}
of $h$ for $\mu$. The smallest possible constant $C$ is called the
\emph{optimal Poincar\'e constant of $\mu$} and is denoted
$C_\textsc{PI}(\mu)$ with the convention $\inf\emptyset=\infty$. Similarly,
$\mu$ satisfies a (Gross) logarithmic Sobolev inequality of constant
$C\in(0,\infty)$ when
\begin{equation}\label{eq:def-GI}
  \mathbf{Ent}_\mu(h^2)\leq C\,\bE(|\nabla h|^2)
\end{equation}
for every smooth function $f:\dR^d\to\dR$, where
$\mathbf{Ent}_\mu(h^2)=\bE_\mu(h^2\log h^2)-\bE_\mu(h^2)\log\bE_\mu(h^2)$ is
the \emph{entropy} or \emph{free energy} of $h^2$ for $\mu$, with the
convention $0\log(0)=0$. As for the Poincar\'e inequality, the smallest
possible $C$ is the \emph{optimal logarithmic Sobolev constant of $\mu$} and
is denoted $C_\textsc{GI}(\mu)$ with $\inf\emptyset=\infty$. Standard
linearization arguments give that
\begin{equation}\label{eq:K-PI-GI}
  \rho(K_\mu) %
  \leq C_\textsc{PI}(\mu) %
  \leq \frac{1}{2}\,C_\textsc{GI}(\mu)
\end{equation}
where $\rho(K_\mu)$ stands for the spectral radius of the covariance matrix
$K_\mu$ of $\mu$ defined by
$(K_\mu)_{i,j}=\bE_\mu(x_ix_j)-\bE_\mu(x_i)\bE_\mu(x_j)$ where $x_i$ and $x_j$
are the coordinate functions. More precisely, the first inequality in
\eqref{eq:K-PI-GI} follows from \eqref{eq:def-PI} by taking $h=\DP{\cdot,u}$
where $u$ runs over the unit sphere while the second inequality in
\eqref{eq:K-PI-GI} follows by considering the directional derivative of both
sides of \eqref{eq:def-GI} at the constant function $1$.

A basic example is given by Gaussian laws for which equalities are achieved in
\eqref{eq:K-PI-GI}. A wide class of laws satisfy Poincar\'e and logarithmic
Sobolev inequalities. Beyond Gaussian laws, a criterion due to Bakry \&
\'Emery \cite{MR889476,MR2002g:46132} (see also \cite{MR2291434},
\cite{MR0109101,MR1132315}, and \cite{caffarelli-fkg,caffarelli-fkg-bis})
states that if $\mu$ has Lebesque density $e^{-V}$ on $\dR^d$ such that
$x\mapsto V(x)-\frac{1}{2\kappa}|x|^2$ is convex for some fixed real
$\kappa>0$ then $C_\textsc{PI}(\mu)\leq \kappa$ and $C_\textsc{GI}(\mu)\leq
2\kappa$ with equality in both cases when $\mu$ is Gaussian. This log-concave
criterion appears as a comparison with Gaussians. Note that in general,
$C_\textsc{GI}(\mu)<\infty$ implies $C_\textsc{PI}(\mu)<\infty$ but the
converse is false. For instance, the law with density proportional to
$\exp(-|x|^a)$ on $\dR$ satisfies a Poincar\'e inequality iff $a\geq1$ and a
logarithmic Sobolev inequality iff $a\geq2$, see e.g. \cite[Chapter
6]{MR2002g:46132}. Note also that if $\mu$ has disconnected support, then
necessarily $C_\textsc{PI}(\mu)=C_\textsc{GI}(\mu)=\infty$. To see it,
consider a non constant $h$ which is constant on each connected component of
the support of $\mu$. This is for instance the case for the two-component
mixture $\mu=p\mu_1+q\mu_0=\cM(p\de_1+q\de_0,\{\mu_0,\mu_1\})$ with
$p\in(0,1)$ and $q=1-p$ where $\mu_0$ and $\mu_1$ have disjoint supports.

The logarithmic Sobolev inequality \eqref{eq:def-GI} implies a sub-Gaussian
concentration of measure for Lipschitz images of $\mu$. Namely, using
\eqref{eq:def-GI} with $h=\exp(\frac{1}{2}\lambda f)$ for a real number
$\lambda$ and a smooth Lipschitz function $f:\dR^d\to\dR$ gives via
Rademacher's Theorem and a standard argument attributed to Herbst
\cite[Chapter 5]{ledoux-ams} that for any reals $\la$ and $r>0$
\begin{equation}\label{eq:GI-to-alpha-beta}
  \al_\mu(\lambda)\leq \frac{1}{4}C_\textsc{GI}(\mu)\la^2 %
  \quad\text{and}\quad %
  \be_\mu(r) \leq 2\exp\PAR{-\frac{r^2}{C_\textsc{GI}(\mu)}}.
\end{equation}
The same method yields from \eqref{eq:def-PI} a sub-exponential upper bound
for $\be_\mu$ of the form $c_1\exp(-c_2r)$ for some constants $c_1,c_2>0$, see
for instance \cite{MR708367} and \cite[Section 2.5]{MR2002j:60002}.

Both Poincar\'e and logarithmic Sobolev inequalities are invariant by the
action of the translation group and the orthogonal group. More generally, let
us denote by $f\cdot\mu$ the image measure of $\mu$ by the map $f$. Both
\eqref{eq:def-PI} and \eqref{eq:def-GI} are stable by Lipschitz maps in the
sense that $C_\textsc{PI}(f\cdot\mu)\leq \NRM{f}^2_\textsc{Lip}
C_\textsc{PI}(\mu)$ and $C_\textsc{GI}(f\cdot\mu)\leq \NRM{f}^2_\textsc{Lip}
C_\textsc{GI}(\mu)$. On the real line, $C_\textsc{PI}$ and $C_\textsc{GI}$ can
be controlled via ``simple'' variational bounds such as \eqref{eq:hardy}. Both
\eqref{eq:def-PI} and \eqref{eq:def-GI} are also stable by bounded
perturbations on the log-density of $\mu$, see \cite{MR893137},
\cite{MR1936110}, and \cite{MR2317340} for further details. In view of
sub-exponential or sub-Gaussian concentration bounds, the main advantage of
\eqref{eq:def-PI} and \eqref{eq:def-GI} over a direct approach based on
$\al_\mu$ or $\be_\mu$ lies in the stability by tensor products of
\eqref{eq:def-PI} and \eqref{eq:def-GI}, see e.g. \cite[Chapters 1 and
3]{MR2002g:46132}, \cite{bolley-villani}, and \cite{gozlan-t2}.

\textbf{The case of mixtures.} The integral criterion \eqref{eq:int-expsq}
shows that if the components of a mixture satisfies uniformly a sub-Gaussian
concentration of measure for Lipschitz functions, and if the mixing law has
compact support, then the mixture also satisfies a sub-Gaussian concentration
of measure for Lipschitz functions. Such bounds appear for instance in
\cite{MR2094433}. However, this observation does not give any fine
quantitative estimate on the dependency over the weights for a finite mixture.
Regarding Poincar\'e and logarithmic Sobolev inequalities, it is clear that a
finite mixture of Gaussians will satisfies such inequalities since its
log-density is a bounded perturbation of a uniformly concave function. Here
again, this does not give any fine control on the constants.
% Notice also that if the union of the supports of the mixture components is
% not connected, then the mixture cannot satisfy a Poincar\'e or logarithmic
% Sobolev inequality, even if each component does.

An upper bound for the Poincar\'e constant of univariate finite Gaussian
mixture was provided by Johnson \cite[Theorem 1.1 and Section 2]{MR2141356}.
Unfortunately, this upper bound blows up when the minimum weight of the mixing
law goes to $0$. A more general upper bound for finite mixtures of overlapping
densities was obtained by Madras and Randall \cite[Theorem 1.2 and Section
5]{MR1910641}. Here again, the bound blows up when the minimum weight of the
mixing law goes to $0$. Some aspects of Poisson mixtures are considered by
Kontoyannis and Madiman \cite{MR2219345,Kont-Madi} in connection with compound
Poisson processes and discrete modified logarithmic Sobolev inequalities.

\textbf{Outline of the article.} Recall that the aim of the present work is to
study fine properties of mixture of law with respect to concentration of
measure and Sobolev type functional inequalities. The analysis of various
elementary examples shows actually that such a general objective is very
ambitious. Also, we decided to focus in the present work on more tractable
situations. Section \ref{se:laplace} provides %general
Laplace bounds for Lipschitz functions in the case of generic mixtures. These
upper bounds on $\al_\mu$ (and thus $\be_\mu$) for a mixture $\mu$ involve the
$W_1$-diameter (see Section \ref{se:laplace} for a precise definition) of the
mixed family. Section \ref{se:conc-two-comp} is devoted to upper bounds on
$\al_\mu$ for two-component mixtures $\mu=\mu_p=p\mu_1+q\mu_0$. Our result is
mainly based on a Laplace-Lipschitz counterpart of the optimal logarithmic
Sobolev inequality for asymmetric Bernoulli measures. In particular, we show
that if $\mu_0$ and $\mu_1$ satisfy a sub-Gaussian concentration for Lipschitz
functions, then it is also the case for the mixture $\mu_p$, with a quite
satisfactory and intuitive behavior as $\min(p,q)$ goes to $0$. In Section
\ref{se:PI-GI-two-comp}, we study Poincar\'e and logarithmic Sobolev
inequalities for two components mixtures. A decomposition of variance and
entropy allows to reduce the problem to the Poincar\'e and logarithmic Sobolev
inequalities for each component, to discrete inequalities for the Bernoulli
mixing law $p\de_1+q\de_0$, and to the control of a mean-difference term. This
last term can be controlled in turn by using some support-constrained
transportation, leading to very interesting open questions in dimension $>1$.
The Poincar\'e constant of the two-component mixture can remain bounded as
$\min(p,q)$ goes to $0$, while the logarithmic Sobolev constant may
surprisingly blow up at speed $-\log(\min(p,q))$. This counter-intuitive
result shows that as far as mixture of laws are concerned, the logarithmic
Sobolev inequality does not behave like the sub-Gaussian concentration for
Lipschitz functions. We also illustrate our results on a gallery of concrete
two-component mixtures. In particular, we show that the blow up of the
logarithmic Sobolev constant as $\min(p,q)$ goes to $0$ is not necessarily
related to support problems.

\textbf{Open problems.} The study of Poincar\'e and logarithmic Sobolev
inequalities for multivariate or non-finite mixtures is an interesting open
problem, for which we give some clues at the end of Section
\ref{se:PI-GI-two-comp} in terms of support-constrained transportation
interpolation. There is maybe a link with the decomposition approach used in
\cite{MR2099650} for Markov chains. One can also explore the tensor products
of mixtures, which are again mixtures. Another interesting problem is the
development of a direct approach for transportation cost and
measure-capacities inequalities (see \cite{br}) for mixtures, even in the
finite univariate case.

\section{General Laplace bounds for Lipschitz functions}
\label{se:laplace}

Intuitively, the concentration of measure of a finite mixture may be
controlled by the worst concentration of the components and some sort of
diameter of the mixed family. We shall confirm, extend, and illustrate this
intuition for a non necessarily finite mixture. The notion of diameter that we
shall use is related to coupling and transportation cost. Recall that for
every $k\geq1$, the Wasserstein (or transportation cost) distance of order $k$
between two laws $\mu_1$ and $\mu_2$ on $\dR^d$ is defined by (see e.g.
\cite{villani-livre,villani-saint-flour} and \cite{MR1105086,MR530375})
\begin{equation}\label{eq:def-wk}
  W_k(\mu_1,\mu_2) %
  = \inf_\pi\PAR{\int_{\dR^d\times\dR^d}\!|x-y|^k\,d\pi(x,y)}^{k^{-1}}
\end{equation}
where $\pi$ runs over the set of laws on $\dR^d\times\dR^d$ with marginals
$\mu_1$ and $\mu_2$. The $W_k$-convergence is equivalent to the weak
convergence together with the convergence of moments up to order $k$. In
dimension $d=1$, we have, by denoting $F_1$ and $F_2$ the cumulative
distribution functions of $\mu_1$ and $\mu_2$, with generalized inverses
$F_1^{-1}$ and $F_2^{-1}$, for every $k\geq1$,
\begin{equation}\label{eq:W-dim1}
  W_k(\mu_1,\mu_2)^k=\!\int_0^1\!\ABS{F_1^{-1}(x)-F_2^{-1}(x)}^k\,dx %
  \ \ \text{and}\ \ %
  W_1(\mu_1,\mu_2) = \!\int_{\dR}\!\ABS{F_1(x)-F_2(x)}\,dx
\end{equation}
where the last expression of $W_1$ follows from the Kantorovich-Rubinstein
dual formulation
\begin{equation}\label{eq:kr-w1}
  W_1(\mu_1,\mu_2) %
  =\sup_{\LIP{f}\leq 1}\PAR{\int_{\dR^d}\!f\,d\mu_1-\int_{\dR^d}\!f\,d\mu_2}.
\end{equation}
Note that if $\mu_1$ does not give mass to points then $\mu_2=(F_2^{-1}\circ
F_1)\cdot\mu_1$. The transportation cost distances lead to the so called
\emph{transportation cost inequalities}, popularized by Marton
\cite{MR838213,MR1404531}, Talagrand \cite{MR1392331}, and Bobkov \& G\"otze
\cite{MR1682772}. See for instance the books
\cite{ledoux-ams,villani-livre,villani-saint-flour} for a review. The link
with concentration of measure was recently deeply explored by Gozlan, see e.g.
\cite{gozlan-t2}. We will not use this interesting line of research in the
present paper.

\begin{thm}[General Laplace-Lipschitz bound via diameter]\label{th:borne}
  Let $\mu=\cM(\nu,{(\mu_\theta)}_{\theta\in\Theta})$ be a general mixture. If
  this mixture satisfies the uniform bounds
  $$
  \overline \alpha =\sup_{\theta\in\Theta}\alpha_\theta<\infty
  \quad\text{and}\quad
  \overline W
  =\sup_{\theta,\theta'\in\Theta}W_1(\mu_\theta,\mu_{\theta'})<\infty
  $$
  then for every $\lambda >0$ we have
  $$
  \alpha_\mu(\lambda) %
  \leq \overline \alpha(\lambda)%
  +\frac{1}{8}\min\PAR{8\overline W\lambda,\overline{W}^2\lambda^2}.
  $$  
\end{thm}

\begin{proof}[Proof of Theorem \ref{th:borne}]
  The key point is that if $\LIP{f}\leq1$ then for every $\lambda>0$,
  \begin{equation}
    \label{eq:genconc}
    \frac{\bE_\mu\PAR{e^{\lambda f}}}{e^{\lambda \bE_\mu f}}%
    = e^{-\lambda
      \bE_\mu f}\int_\Theta\!\bE_{\mu_\theta}\PAR{e^{\lambda
        f}}\,\nu(d\theta) \leq \int_\Theta\! e^{\alpha_\theta(\lambda)+
      \lambda(\bE_{\mu_\theta}f-\bE_\mu f)}\,\nu(d\theta).
  \end{equation}
  As a consequence, we get 
  \begin{equation}\label{eq:majal}
    \alpha_\mu(\lambda) %
    \leq \overline\alpha(\lambda) %
    +\sup_{\LIP{f}\leq1}\log\int_\Theta\!e^{\lambda(\bE_{\mu_\theta}f-\bE_\mu f)}\,\nu(d\theta).
  \end{equation}
  Thanks to the relation \eqref{eq:kr-w1}, we obtain
  \begin{align*}
    \bE_{\mu_\theta}f-\bE_\mu f
    &=\int_{\Theta}\!\PAR{\bE_{\mu_\theta}f-\bE_{\mu_{\theta'}}f}\,\nu(d\theta')\\
    &\leq \int_{\Theta}\!W_1(\mu_\theta,\mu_{\theta'})\,\nu(d\theta')
    \leq \overline W.
  \end{align*}
  This shows that the second term in the right hand side of \eqref{eq:majal}
  is bounded by $\overline W\lambda$. Alternatively, one can use the Hoeffding
  bound \cite{hoeffding} which says that if $X$ is a centered bounded random
  variable with oscillation $c=\sup X-\inf X$ then
  $$ 
  \bE\PAR{ e^{\lambda X}}\leq e^{\frac{1}{8}\lambda^2c^2}.
  $$
  The desired bound in terms of $\overline W^2\lambda^2$ follows by taking
  $X=\bE_{\mu_Y}f-\bE_\mu f$ where $Y\sim\nu$ and noticing that
  $c\leq\sup_{\theta,\theta'}\PAR{\bE_{\mu_\theta}f-\bE_{\mu_{\theta'}}f}=\overline
  W$.
\end{proof}

\begin{xpl}[Finite mixtures]
  For a finite mixture $\mu=p_1\mu_1+\cdots+p_n\mu_n=\cM(\nu,(\mu_i)_{1\leq
    i\leq n})$ where $\nu=p_1\de_1+\cdots+p_n\de_n$, the mixing measure $\nu$
  is supported by a finite set. In that case, Theorem \ref{th:borne} gives an
  immediate Laplace bound, involving the worst bound for the mixture
  components $(\mu_i)_{1\leq i\leq n}$ (this cannot be improved in general).
  However, in Section \ref{se:conc-two-comp}, we provide sharper bounds by
  improving the dependency over $\nu$ in the case where $n=2$.
\end{xpl}

\begin{xpl}[Bounded mixtures of multivariate Gaussians]
  Here $\mu_\theta=\cN(m(\theta),\Gamma(\theta))$ where $m:\Theta\to\dR^d$ and
  $\Gamma:\dR^d\to\cS^+_d$ are two measurable bounded functions and $\cS^+_d$
  is the cone of symmetric nonnegative $d\times d$ matrices. Note that
  $\Gamma(\te)$ is allowed to be singular \emph{i.e.} not of full rank. The
  spectrum of $\Gamma(\te)$ is real and non-negative. If
  $\lambda_1(\theta)\geq\cdots\geq\lambda_d(\theta)$ are the eigenvalues of
  $\Gamma(\theta)$, we define
  $\rho=\sup_{\te\in\Te}\lambda_1(\te)=\sup_{\te\in\Te}\NRM{\Gamma(\te)}_{2\to2}$.
  Now fix some mixing law $\nu$ on $\Theta$ and consider the mixture
  $\mu=\cM(\nu,{(\mu_\theta)}_{\theta\in\Theta})$. Then for every $\lambda>0$,
  $$
  \alpha_\mu(\lambda) %
  \leq \frac{\rho}{2}\lambda ^2 %
  +\frac{1}{8}\min(8\overline W\lambda,\overline{W}^2\lambda^2).
  $$
  One can deduce an upper bound for $\overline{W}$ from the following lemma.
  \end{xpl}

  \begin{lem}[$W_1$-distance of two multivariate Gaussian laws]\label{le:w1g}
    Let $\mu_0=\cN(m(0),\Gamma(0))$ and $\mu_1=\cN(m(1),\Gamma(1))$ be two
    Gaussian laws on $\dR^d$. For $\theta\in\{0,1\}$, we denote by
    $$
    \lambda_1(\theta)\geq\cdots\geq\lambda_d(\theta)
    $$
    the ordered spectrum of $\Gamma(\theta)$ and by ${(v_i{(\theta)})}_{1\leq
      i\leq d}$ an associated orthonormal basis of eigenvectors. Assume,
    without loss of generality, that $v_i{(0)}\cdot v_i{(1)}\geq 0$ for every
    $1\leq i\leq d$ where ``$\cdot$'' stands for the Euclidean scalar product
    of $\dR^d$. Then $W_1(\mu_0,\mu_1)$ is bounded above by
    $$
    \ABS{m(1)-m(0)}+
    \sqrt{\sum_{i=1}^d\BRA{\PAR{\sqrt{\lambda_i{(1)}}-\sqrt{\lambda_i{(0)}}}^2
        +2\sqrt{\lambda_i{(1)}\lambda_i{(0)}}\PAR{1-v_i{(1)}\cdot v_i{(0)}}}}.
    $$
  \end{lem}

  The reader may find in \cite[Theorem 3.2]{takatsu} a formula in the same
  spirit for $W_2(\mu_0,\mu_1)$.

  \begin{proof}[Proof of Lemma \ref{le:w1g}]
    The triangle inequality for the $W_1$ distance gives
    \begin{align*}
      W_1(\mu_0,\mu_1)
      &\leq W_1(\mu_0,\cN(m(1),\Gamma(0))) + W_1(\cN(m(1),\Gamma(0)),\mu_1)\\
      &\leq \ABS{m(1)-m(0)}+W_1(\cN(0,\Gamma(0)),\cN(0,\Gamma(1))).
    \end{align*}
    Now, if ${(Y_i)}_{1\leq i\leq d}$ are i.i.d.\ real random variables of law
    $\cN(0,1)$ then the law of
    $$
    X_\theta=\sum_{i=1}^d Y_i\sqrt{\lambda_i{(\theta)}}v_i{(\theta)}
    $$
    is $\cN(0,\Gamma(\theta))$ for $\theta\in\{0,1\}$. Moreover, from
    \eqref{eq:def-wk} and Jensen's inequality, we get
    $$
    W_1(\cN(0,\Gamma(0)),\cN(0,\Gamma(1)))^2
    \leq \PAR{\dE\ABS{X_1-X_0}}^2
    \leq \dE(\ABS{X_1-X_0}^2).
    $$
    At this step, we note that
    \begin{align*}
      \ABS{X_1-X_0}^2=
      &\sum_{i=1}^dY_i^2\ABS{\sqrt{\lambda_i(1)}v_i(1)-\sqrt{\lambda_i(0)}v_i(0)}^2\\
      &+2\sum_{i<j}Y_iY_j\PAR{\sqrt{\lambda_i(1)}v_i(1)-\sqrt{\lambda_i(0)}v_i(0)}\cdot\PAR{\sqrt{\lambda_i(1)}v_i(1)-\sqrt{\lambda_i(0)}v_i(0)}.
    \end{align*}
    Since ${(Y_i)}$ are i.i.d.\ $\cN(0,1)$ and ${(v_i(\theta))}_{1\leq i\leq
      d}$ is orthonormal for $\theta\in\{0,1\}$, one has
    \begin{align*}
      \dE(\ABS{X_1-X_0}^2)
      &=\sum_{i=1}^d\ABS{\sqrt{\lambda_i(1)}v_i(1)-
        \sqrt{\lambda_i(0)}v_i(0) }^2\\
      &=\sum_{i=1}^d\BRA{\PAR{\sqrt{\lambda_i{(1)}}-\sqrt{\lambda_i{(0)}}}^2
        +2\sqrt{\lambda_i{(1)}\lambda_i{(0)}}\PAR{1-v_i{(1)}\cdot v_i{(0)}}}.
    \end{align*}
  \end{proof}

Of course the assumptions of Theorem \ref{th:borne} may be relaxed. Instead of
trying to deal with generic abstract results, let us provide some highlighting
examples.

\begin{xpl}[Gaussian mixture of translated Gaussians]
  Here $\Theta=\dR$ and $\mu_\theta=\cN(\theta,\sigma^2)$ for some fixed
  $\sigma>0$, and the mixing law is also Gaussian $\nu=\cN(0,\tau^2)$ for some
  fixed $\tau>0$. In this case, $\overline
  \alpha(\lambda)=\frac{1}{2}\sigma^2\lambda^2$ but $\overline W$ is infinite
  since 
  $$
  W_1(\mu_\theta,\mu_{\theta'})=\ABS{\theta-\theta'}.
  $$
  In particular, Theorem \ref{th:borne} is useless. Nevertheless, the function
  $$
  \theta\mapsto g(\theta)=\bE_{\mu_\theta} f-\bE_\mu f
  $$
  is Lipschitz since
  $$
  \ABS{g(\theta)-g(\theta')}\leq \bE\PAR{\ABS{f(X+\theta)-f(X+\theta')}} %
  \leq \ABS{\theta-\theta'}
  $$
  where $X\sim\cN(0,1)$. As a consequence, we get
  $$
  \sup_{\LIP{f}\leq 1}\log%
  \int_\Theta\!e^{\lambda(\bE_{\mu_\theta} f-\bE_\mu f)}\,\nu(d\theta)%
  \leq \frac{\tau^2\lambda^2}{2},
  $$
  and for any $\lambda>0$
  $$
  \alpha_\mu(\lambda)\leq \frac{\sigma^2+\tau^2}{2}\lambda^2.
  $$
  The same argument may be used more generally for ``position'' mixtures. For
  instance if $\eta$ is some fixed probability measure on $\dR^d$ and
  $\mu_\theta=\eta*\delta_\theta$ for $\theta\in\dR^d$ then $\forall
  \lambda>0$,
  $$
  \alpha_\mu(\lambda)\leq \alpha_\eta(\lambda)+\alpha_\mu(\lambda).
  $$
  In this particular case, $\mu=\nu*\eta$ and the bound above follows also by
  tensorization!
\end{xpl}

\begin{xpl}[Mixture of scaled Gaussians: from exponential to Gaussian tails]
  Here we take $\Theta=[0,\infty)$ and $\mu_\theta=\cN(0,\theta^2)$ with a
  mixing measure $\nu$ of density 
  $$
  \theta\mapsto \frac{\gamma}{\Gamma\!\PAR{\gamma^{-1}}} %
  \exp\PAR{-\theta^{\gamma}}\mathds{1}_{[0,\infty)}(\theta)
  $$
  where $\gamma\geq2$ is some fixed real number. Note that $\nu$ has a
  non-compact support and that $\mu$ does not satisfy the integral criterion
  \eqref{eq:int-expsq}. %Guillin-Guillin!
  This means that $\mu$ cannot have sub-Gaussian tails. Note also that both
  $\overline \alpha(\lambda)$ and $\overline W$ are infinite since
  $$
  \alpha_\theta(\lambda)=\frac{\theta^2\lambda^2}{2} %
  \quad\text{and}\quad %
  W_1(\mu_\theta,\mu_{\theta'})=\sqrt{\frac{2}{\pi}}\,\ABS{\theta-\theta'}
  $$
  where we used \eqref{eq:W-dim1} for $W_1$. Starting from \eqref{eq:genconc},
  one has by Cauchy-Schwarz's inequality
  \begin{equation}\label{eq:melvar}
    \PAR{\frac{\bE_\mu\PAR{e^{\lambda f}}}{e^{\lambda \bE_\mu f}}}^2
    \leq \int_\Te\!e^{\theta^2\lambda^2}\,\nu(d\theta)%
         \int_\Te\!e^{2\lambda(\bE_{\mu_\theta} f-\bE_\mu f)}\,\nu(d\theta).
  \end{equation}
  Note that $\nu$ satisfies condition \eqref{eq:int-expsq} and
  $\alpha_\nu(\lambda)\leq C\lambda^2$ for some real constant $C>0$. Here and
  in the sequel, the constant $C$ may vary from line to line and may be chosen
  independent of $\gamma$. On the other hand, the centered function
  $g(\theta)=\bE_{\mu_\theta} f-\bE_\mu f$ is 1-Lipschitz since
  $$
  \ABS{g(\theta)-g(\theta')}=\ABS{\bE f(\theta X)-\bE f(\theta' X)}
  \leq \ABS{\theta-\theta'}\bE(\ABS{X})
  $$
  where $X\sim\cN(0,1)$. Also, for the second term in the right hand side of
  \eqref{eq:melvar} we have
  $$
  \int_\Te\!e^{2\lambda(\bE_{\mu_\theta} f-\bE_\mu f)}\,\nu(d\theta) %
  \leq e^{\alpha_\nu(2\lambda)} %
  \leq e^{4C\lambda^2}.
  $$
  If $\gamma=2$ then $\alpha_\mu(\lambda)\leq
  2C\lambda^2-\frac{1}{4}\log(1-\lambda^2)\leq 2C-\frac{1}{4}\log(1-\lambda)$
  if $\lambda<1$, which gives, after some computations, the deviation bound,
  for some other constants $C'>0$ and $C''>0$,
  $$
  \mu(F-\bE_\mu f\geq r )\leq C'e^{-C''r}.
  $$
  Assume in contrast that $\gamma>2$. Since $
  \theta^2\lambda^2\leq\gamma^{-1}\theta^\gamma
  +C_0\lambda^{\frac{2\gamma}{\gamma-2}}$ for some constant $C_0>0$ which may
  depend on $\gamma$ but not on $\lambda$ and $\theta$, we get, for some
  constants $C_1>0$ and $C_2>0$,
  $$
  \int_0^\infty\!\exp\PAR{\theta^2\lambda^2}\,\nu(d\theta) %
  \leq C_1\exp\PAR{C_2\lambda^{\frac{2\gamma}{\gamma-2}}}.
  $$
  This gives
  $\alpha_\mu(\lambda)\leq{}C_3\lambda^{\frac{2\gamma}{\gamma-2}}+C_4$ for
  some constants $C_3>0$ and $C_4>0$, which yields a deviation bound of the
  form (for some constants $C_5>0$ and $C_6>0$)
  $$
  \mu(f-\bE_\mu f\geq r)\leq %
  C_5\exp\PAR{-C_6r^{2-\frac{4}{\gamma+2}}}.
  $$
  Note that $\nu$ goes to the uniform law on $[0,1]$ as $\gamma\to\infty$ and
  the Gaussian tail reappears.
\end{xpl}

\section{Concentration bounds for two-component mixtures}
\label{se:conc-two-comp}

In this section, we investigate the special case where the mixing measure
$\nu$ is the Bernoulli measure $\cB(p)=p\delta_1+q\delta_0$ where $q=1-p$. We
are interested in the study of the sharp dependence of the concentration
bounds on $p$, especially when $p$ is close to $0$ or $1$.

\begin{thm}[Two-component mixture]\label{th:two-mix-lap}
  Let $\mu_0$ and $\mu_1$ be two probability measures on $\cX$ and
  $\mu=p\mu_1+q\mu_0$ with $p\in[0,1]$ and $q=1-p$. Define 
  $x_p=\max(p,q)/(2c_p)$ where
  $$
  c_p=\frac{q-p}{4(\log(q)-\log(p))}
  $$
  with the continuity conventions $c_{1/2}=1/8$ and $c_0=c_1=0$. Then for any
  $\la>0$,
  $$
  \al_{\mu}(\la)
  \leq \max(\al_{\mu_0},\al_{\mu_1})(\la)
  +\begin{cases}
    \ds{c_p\la^2W_1(\mu_0,\mu_1)^2} &\text{if $\la W_1(\mu_0,\mu_1)\leq
      x_p$}\\
    \\
    \ds{\max(p,q)\PAR{\la W_1(\mu_0,\mu_1)-\frac{1}{2}x_p}} & \text{otherwise}.
  \end{cases}
  $$
  % $$
  % \al_{\mu}(\la)
  % \leq \max(\al_{\mu_0},\al_{\mu_1})(\la)
  % +\max(p,q)W_1(\mu_0,\mu_1)\la,
  % $$
  % and
  % $$
  % \al_{\mu}(\la)
  % \leq \max(\al_{\mu_0},\al_{\mu_1})(\la)
  % +c_pW_1(\mu_0,\mu_1)^2\la^2.
  % $$
\end{thm}

Note that if $\min(p,q)\to0$, then $c_p\sim-(4\log(p))^{-1}\to0$ and
$x_p\to\infty$, and we thus recover an upper bound of the form $\al_\mu\leq
\max(\al_{\mu_1},\al_{\mu_2})$ as $\min(p,q)\to0$, which is satisfactory. The
two different upper bounds given by Theorem \ref{th:two-mix-lap} provide two
different upper bounds for the concentration of measure of the mixture $\mu$,
illustrated by the following Corollary (the proof of the Corollary is
immediate and is left to the reader).

\begin{cor}[Two-component mixtures with sub-Gaussian tails]\label{co:gaustail}
  Let $\mu_0$ and $\mu_1$ be two probability measures on $\cX$ and
  $\mu=p\mu_1+q\mu_0$ for some $p\in[0,1]$ with $q=1-p$. If there exists a
  real constant $C>0$ such that for any $\la>0$
  $$
  \max(\al_{\mu_0},\al_{\mu_1})(\la)\leq \frac{1}{2}C\la^2
  $$
  then for every $r\geq0$, with $\overline W=W_1(\mu_0,\mu_1)$,
  $$
  \be_\mu(r) \leq
  2
  \begin{cases}
   \ds{\exp\PAR{-\frac{r^2}{2C+4c_p\OL{W}^2}}}
   & \text{if $r\leq \max(p,q)\PAR{\frac{C}{2c_p\OL{W}}+\OL{W}}$}, \\
   & \\
   \ds{\exp\PAR{-\frac{1}{2C}(r-\max(p,q)\OL{W})^2-\frac{\max(p,q)^2}{4c_p}}}
   & \text{otherwise.}
   \end{cases}
  $$
  % then for every $r\geq \max(p,q)W_1(\mu_0,\mu_1)$,
  % $$
  % \be_\mu(r) \leq 2\exp\PAR{-\frac{1}{2C}(r-\max(p,q)W_1(\mu_0,\mu_1))^2},
  % $$
  % and for every $r\geq 0$,
  % $$
  % \be_\mu(r) \leq 2\exp\PAR{-\frac{r^2}{2C+4c_pW_1(\mu_0,\mu_1)^2}}.
  % $$
\end{cor}

\begin{proof}[Proof of Theorem \ref{th:two-mix-lap}]
  We have $\mu=q\mu_0+p\mu_1=\cM(\nu,\{\mu_0,\mu_1\})$ where
  $\nu:=q\de_0+p\de_1$. For this finite mixture, we get, as in the general
  case, for any $f\in\mathrm{Lip}(\cX,\dR)$ and $\la>0$,
  $$
  \log\PAR{\frac{\bE_\mu\PAR{e^{\la f}}}{e^{\la\bE_\mu f}}}
  \leq \max(\al_{\mu_0},\al_{\mu_1})(\la)
  +\log\PAR{\frac{\bE_\nu\PAR{e^{\la g}}}{e^{\la\bE_\nu g}}},
  $$
  where $g(i):=\bE_{\mu_i}f$. At this step, we use the particular nature of
  $\nu$, which gives
  $$
  \frac{\bE_\nu\PAR{e^{\la g}}}{e^{\la\bE_\nu g}} %
  =\cosh_p(\la(g(1)-g(0))),
  $$
  where $\cosh_{p}(x):=pe^{qx}+qe^{-px}$. Since $g(1)-g(0) =
  \bE_{\mu_1}f-\bE_{\mu_0}f$, we get by \eqref{eq:kr-w1}
  $$
  -W_1(\mu_0,\mu_1) \leq g(1)-g(0) \leq W_1(\mu_0,\mu_1).
  $$
  Since $\cosh_p(-x)=\cosh_q(x)$ for any $x\in\dR$, we get for any $\la>0$,
  $$
  \sup_{\LIP{f}\leq 1} 
  \PAR{\frac{\bE_\nu\PAR{e^{\la g}}}{e^{\la\bE_\nu g}}} %
  = \max\PAR{\cosh_{p},\cosh_{q}}(\la W_1(\mu_0,\mu_1)).
  $$
  Putting all together, we obtain, for any $\la>0$,
  $$
  \al_\mu(\la) \leq \max(\al_{\mu_0},\al_{\mu_1})(\la) %
  +\log\max\PAR{\cosh_{p},\cosh_{q}}(\la W_1(\mu_0,\mu_1)),
  $$
  Since $(\cosh_q-\cosh_p)'(x)=2pq(\cosh(px)-\cosh(qx))$, one has, for
  every $x\geq 0$,
  $$
  \max\PAR{\cosh_{p},\cosh_{q}}(x)= \cosh_{\min(p,q)}(x).
  $$
  Let us assume that $p\leq q$. Lemma \ref{le:cp} ensures that, for every
  $x\geq 0$, 
  $$
  \log\max\PAR{\cosh_{p},\cosh_{q}}(x)= \log\cosh_{p}(x)\leq c_p x^2.
  $$
  On the other hand,
  $$
  \log\cosh_p(x)=qx+\log\PAR{p+qe^{-x}}\leq qx.
  $$
  Now, for $x=x_p$, the slope of $x\mapsto c_px^2$ is equal to $q$ and the
  tangent is $y=q(x-x_p/2)$. On the other hand, the convexity of
  $x\mapsto\log\cosh_p(x)$ yields $\log\cosh_p(x)\leq q(x-x_p)$ for $x\geq
  x_p$ (drawing a picture may help the reader). The desired conclusion follows
  immediately.
\end{proof}

The proof of Theorem \ref{th:two-mix-lap} relies on Lemma \ref{le:cp} below
which provides a Gaussian bound for the Laplace transform of a Lipschitz
function with respect to a Bernoulli law. This lemma is an optimal version of
the Hoeffding bound \cite{hoeffding} in the case of a Bernoulli law.

\begin{lem}[Two-point lemma]\label{le:cp}
  For any $0\leq p \leq 1/2$, we have
  \begin{equation}
    \label{eq:cp}
    \sup_{x>0} x^{-2}\log(pe^{qx}+qe^{-px})
    =c_p=\frac{q-p}{4(\log(q)-\log(p))}
  \end{equation}
  with the natural conventions $c_0=0$ and $c_{1/2}=1/8$ as in Theorem
  \ref{th:two-mix-lap}. Moreover, the supremum in $x$ is achieved for
  $x=2(\log(q)-\log(p))$.
\end{lem}

The constant $c_p$ is also equal, as it will appear in the proof, to
$\sup_{\lambda>0}\alpha_{\cB(p)}(\lambda)/\lambda^2$. The classical Hoeffding
bound for this supremum is $c_{1/2}=1/8$ which is the maximum of $c_p$ over
$p$. Additionally, the quantity $1/(4c_p)$ is the optimal constant of the
logarithmic Sobolev inequality for the asymmetric Bernoulli measure
$q\delta_0+p\delta_1$ (see Lemma \ref{le:gi-bern}).

\begin{proof}[Proof of Lemma \ref{le:cp}]
  Let us define $\WH{x}_p=\log(q/p)$ and $\beta(x)=x^{-2}\psi(x)$ where
  $$
  \psi(x)=\log(pe^{qx}+qe^{-px}).
  $$
  The function $\psi$ is ``strongly convex'' at the origin
  ($\psi(0)=\psi'(0)=0$ and $\psi''(0)=pq$ and $\psi'''(0)>0$) and linear at
  infinity ($\psi'(\infty)=q$). Therefore, the supremum of $\beta$ is achieved
  for some $x>0$. The derivative of $\beta$ has the sign of
  $\gamma(x):=x\psi'(x)-2\psi(x)$. Furthermore,
  $$
  \gamma'(x)=x\psi''(x)-\psi'(x) \quad\text{and}\quad
  \gamma''(x)=x\psi'''(x).
  $$
  As a consequence, $\gamma''$ has the sign of $\psi'''$ which is positive on
  $(0,\WH{x}_p)$ and negative on $(\WH{x}_p,+\infty)$. Since $\gamma'(0)=0$
  and $\gamma'$ achieves its maximum for $x=\WH{x}_p$ and $\gamma'$ goes to
  $-q$ at infinity and there exists an unique $y_p>0$ (in fact $y_p>\WH{x}_p$)
  such that $\gamma'(y_p)=0$. As a conclusion, since $\gamma(0)=0$ and
  $\gamma$ is increasing on $(0,y_p)$ and $\gamma(x)$ goes to $-\infty$ as $x$
  goes to infinity, $\gamma(x)$ is equal to zero exactly two times: for $x=0$
  and $x=z_p>y_p>\WH{x}_p$ In fact, $z_p$ is equal to $2\WH{x}_p$. Indeed, we
  have
  $$
  \psi'(x)=pq\frac{e^{qx}-e^{-px}}{pe^{qx}+qe^{-px}}.
  $$
  Now, we compute
  $$
  \psi'(2\WH{x}_p) =pq\frac{(q/p)^{2q}-(p/q)^{2p}}{p(q/p)^{2q}+q(p/q)^{2p}}
  =\cdots=q^2-p^2=q-p,
  $$
  and
  \begin{align*}
    2\psi(2\WH{x}_p)
    &=2\log(p(q/p)^{2q}+q(p/q)^{2p}) \\
    &=2\log((q+p)(q/p)^{q-p}) \\
    &=2\WH{x}_p\psi'(2\WH{x}_p).
  \end{align*}
  Thus, $2\WH{x}_p$ is (the unique positive) solution of $2\psi(x)=x\psi'(x)$.
  As a conclusion, we get $c_p=\psi(2\WH{x}_p)/(4\WH{x}_p^2)$, which gives the
  desired formula after some algebra.
\end{proof}

\begin{rem}[Advantage of direct Laplace bounds]\label{rm:advlap} 
  Consider a mixture $\mu=p\mu_1+q\mu_0$ of two Gaussian laws $\mu_0$ and
  $\mu_1$ on $\dR$ with same variance $\si^2$ and different means. Corollary
  \ref{co:gaustail} ensures that for every $r\geq0$,
  $$
  \be_\mu(r) \leq 2\exp\PAR{-\frac{r^2}{2\si^2+4c_pW_1(\mu_0,\mu_1)^2}}.
  $$
  This bound remains relevant as $\si\to0$ since we recover the bound for the
  Bernoulli mixing law $\nu=p\de_1+q\de_0$. On the other hand, any
  concentration bound deduced from a logarithmic Sobolev inequality would blow
  up as $\si$ goes to zero, as we shall see in Section
  \ref{se:PI-GI-two-comp}.
\end{rem}

\begin{rem}[Inhomogeneous tails]
  It is satisfactory to recover, when $p$ goes to $0$ (resp. $1$), the
  concentration bound of $\mu_0$ (resp. $\mu_1$) and not only the maximum of
  the bounds of the two components. It is possible to exhibit two regimes,
  corresponding to small and big values of $\lambda$. Assume that
  $\mu_i=\cN(0,\theta_i^2)$ for $i\in\{0,1\}$ with $\theta_1>\theta_0>0$.
  Theorem \ref{th:borne} gives
  $$
  \alpha_\mu(\lambda)\leq
  \frac{\theta_1^2\lambda^2}{2}+(\theta_1-\theta_0)\lambda.
  $$
  On the other hand, one has 
  $$
  \log\frac{\bE_\mu\PAR{e^{\la f}}}{e^{\bE_\mu\PAR{\la f}}}
  \leq\int\!\alpha_{\mu_\theta}(\lambda)\,\nu(d\theta) +
  \log\int\!e^{H_\lambda(\theta)+\lambda g(\theta)}\,\nu(d\theta), 
  $$
  where 
  $$
  H_\lambda(\theta)%
  =\alpha_{\mu_\theta}(\lambda)-\int\!\alpha_{\mu_{\theta'}}(\lambda)\,\nu(d\theta')%
  \quad\text{and}\quad%
  g(\theta)=\bE_{\mu_\theta}f-\bE_{\mu}f.
  $$
  Then, Lemma \ref{le:cp} ensures that for every $\varepsilon>0$,
  \begin{align*}
    \log\int\!e^{H_\lambda(\theta)+\lambda g(\theta)}\,\nu(d\theta)
    &\leq c_p \PAR{H_\lambda(1)+\lambda g(1)-H_\lambda(0)-\lambda g(0)}^2\\
    &\leq c_p \PAR{\frac{1}{\varepsilon}\ABS{H_\lambda(1)-H_\lambda(0)}^2%
      +\varepsilon\ABS{\lambda g(1)-\lambda g(0)}^2}.
  \end{align*}
  Choosing $\varepsilon=\lambda$ leads to 
  $$
  \log\int\!e^{H_\lambda(\theta)+\lambda g(\theta)}\,\nu(d\theta) \leq
  c_p\PAR{\frac{(\theta_1^2-\theta_0^2)^2}{4}+(\theta_1-\theta_0)^2}\lambda^3.
  $$ 
  As a conclusion $\alpha_\mu$ can be control by (at least) these two ways:
  $$
  \alpha_\mu(\lambda)\leq
  \begin{cases}
    \ds{\frac{\theta_1^2\lambda^2}{2}+(\theta_1-\theta_0)\lambda},&\\
    \\
    \ds{\frac{p\theta_1^2+q\theta_0^2\lambda^2}{2}+c_p\PAR{\frac{(\theta_1^2-\theta_0^2)^2}{4}+(\theta_1-\theta_0)^2}\lambda^3}.&
  \end{cases}
  $$
  The second one provides sharp bounds for $\lambda\leq f(1/c_p)$ whereas the
  first one is useful for $\lambda\geq f(1/c_p)$ (where $f$ is an increasing
  function which is computable).
\end{rem}

\section{Gross-Poincar\'e inequalities for two-component mixtures}
\label{se:PI-GI-two-comp}

It is known that functional inequalities such as Poincar\'e and (Gross)
logarithmic Sobolev inequalities provide, via Laplace exponential bounds,
dimension free concentration bounds, see for instance \cite{ledoux-ams}. It is
quite natural to ask for such functional inequalities for mixtures. Before
attacking the problem, some facts have to be emphasized.

As already mentioned in the introduction, a law $\mu$ with disconnected
support cannot satisfy a Poincar\'e or a logarithmic Sobolev inequality. In
particular, a mixture of laws with disjoint supports cannot satisfy such
functional inequalities. This observation suggests that in order to obtain a
functional inequality for a mixture, one has probably to control the
considered functional inequality for each component of the mixture and to
ensure that the support of the mixture is connected. It is important to
realize that such a connectivity problem is due to the peculiarities of the
functional inequalities, but does not pose a real problem for the
concentration of measure properties, as suggested by Theorem
\ref{th:two-mix-lap} and Remark \ref{rm:advlap} for instance. In the sequel,
we will focus on the case of two-component mixtures, and try to get sharp
bounds on the Poincar\'e and logarithmic Sobolev constants. The two-component
case is fundamental. The extension of the results to more general finite
mixtures is possible by following roughly the same scheme, see Remark
\ref{rm:fimix} below.

For the logarithmic Sobolev inequality of two-component mixtures, we will make
use of the following optimal two-point Lemma, obtained years ago independently
by Diaconis \& Saloff-Coste and Higushi \& Yoshida. An elementary proof due to
Bobkov is given by Saloff-Coste in his Saint-Flour Lecture Notes
\cite{MR1490046}.

\begin{lem}[Optimal logarithmic Sobolev inequality for Bernoulli measures]
  \label{le:gi-bern}
  For every $p\in(0,1)$ and every $f:\{0,1\}\to\dR$, and with the convention
  $(\log(q)-\log(p))/(q-p)=2$ if $p=q=1/2$, we have
  $$
  \ent{p\delta_1+q\delta_0}{f^2} %
  \leq \frac{\log(q)-\log(p)}{q-p} pq \PAR{f(0)-f(1)}^2.
  $$
  Moreover, the function of $p$ in front of the right hand side cannot be
  improved.
\end{lem}

Note that the constant in front of the right hand side of the inequality
provided by Lemma \ref{le:gi-bern} is nothing else but $pq/(4c_p)$ where $c_p$
is as in Theorem \ref{th:two-mix-lap} and Lemma \ref{le:cp}. At this stage, it
is important to understand the deep difference between the Poincar\'e and the
logarithmic Sobolev inequalities at the level of the two-point space. On the
two-point space, the Poincar\'e inequality turns out to be a simple equality,
and Lemma \ref{le:gi-bern} is in fact an entropy-variance comparison. Namely,
for every $p\in(0,1)$ and $f:\{0,1\}\to\dR$,
$$
\ent{p\de_1+q\de_0}{f^2}\leq\frac{\log(q)-\log(p)}{q-p}\var{p\de_1+q\de_0}{f}.
$$
This inequality is optimal and $(\log(q)-\log(p))/(q-p)$ tends to $+\infty$ as
$\min(p,q)$ goes to $0$. Also, for strongly asymmetric Bernoulli measures, the
entropy of the square can take extremely big values for a fixed prescribed
variance. This elementary phenomenon helps to better understand the surprising
difference in the behavior of the Poincar\'e and logarithmic Sobolev constants
of certain two-component mixtures exhibited in the sequel. Moreover, this
observation suggests to use asymmetric test functions inspired from the
two-point space in order to show that the logarithmic Sobolev constant may
blow up when the mixing law is strongly asymmetric. We shall adopt however
another (quantitative) route.

\subsection{Decomposition of the variance and entropy of the mixture}

Let $\mu_0$ and $\mu_1$ be two laws on $\dR^d$, $p\in [0,1]$, $q=1-p$,
$\nu=p\delta_1+q\delta_0$, and $\mu_p=p\mu_1+q\mu_0$. Then, one can decompose
and bound the variance of $f:\dR^d\to\dR$ with respect to $\mu_p$ as
\begin{align*}
  \var{\mu_p}{f}%
  &=\bE_\nu\PAR{\theta\mapsto\var{\mu_\theta}{f}}%
  +\var{\nu}{\theta\mapsto\bE_{\mu_\theta}f}\\
  &=\bE_\nu\PAR{\theta\mapsto\var{\mu_\theta}{f}}%
  +pq\PAR{\bE_{\mu_0}f-\bE_{\mu_1}f}^2 \\
  &\leq \max(C_\textsc{PI}(\mu_0),C_\textsc{PI}(\mu_1))%
  \bE_\mu(\ABS{\nabla f}^2) %
  +pq\PAR{\bE_{\mu_0}f-\bE_{\mu_1}f}^2.
\end{align*}
For the entropy, let us write 
$$
  \ent{\mu_p}{f^2}%
  =\bE_\nu\PAR{\te\mapsto\ent{\mu_\theta}{f^2}}%
  +\ent{\nu}{\te\mapsto\bE_{\mu_\theta}(f^2)}.
$$
Applying Lemma \ref{le:gi-bern} to the function
$\te\mapsto\sqrt{\bE_{\mu_\theta}(f^2)}$, one gets
$$
\ent{\nu}{\te\mapsto\bE_{\mu_\theta}(f^2)}\leq %
\frac{pq(\log q-\log p)}{q-p}%
\PAR{\sqrt{\bE_{\mu_0}(f^2)}-\sqrt{\bE_{\mu_1}(f^2)}}^2. 
$$
Since
$\bE_{\mu_0}(f)\bE_{\mu_1}(f)\leq\sqrt{\bE_{\mu_0}(f^2)\bE_{\mu_1}(f^2)}$,
we have
\begin{align*}
  \PAR{\sqrt{\bE_{\mu_0}(f^2)}-\sqrt{\bE_{\mu_1}(f^2)}}^2=&%
 \bE_{\mu_0}(f^2)+\bE_{\mu_1}(f^2)-2\sqrt{\bE_{\mu_0}(f^2)\bE_{\mu_1}(f^2)}\\
\leq&\var{\mu_0}{f}+\var{\mu_1}{f}+\PAR{\bE_{\mu_0}f-\bE_{\mu_1}f}^2.  
\end{align*}
(Note that the right hand side is not equal to zero if $\mu_0=\mu_1$).
Using the Poincar\'e inequalities for $\mu_0$ and $\mu_1$ provides the
following control of the entropy:
\begin{align*}
  \ent{\mu_p}{f^2}%
\leq& \max(C_\textsc{GI}(\mu_0),C_\textsc{GI}(\mu_1))%
  \bE_\mu(\ABS{\nabla f}^2)\\%  
&+\frac{pq(\log q-\log p)}{q-p}
\PAR{\bE_{\mu_0}f-\bE_{\mu_1}f}^2\\
&+\max(C_\textsc{PI}(\mu_0),C_\textsc{PI}(\mu_1))%
  \frac{\log q-\log p}{q-p}\bE_{\mu_p}(\ABS{\nabla f}^2). 
\end{align*}
(The worst term is the last one since it always explodes as $\min(p,q)$ goes
to zero).
We thus see that in both cases (Poincar\'e and logarithmic Sobolev
inequalities), the problem can be reduced to the control of the
mean-difference term $\PAR{\bE_{\mu_0}f-\bE_{\mu_1}f}^2$ in terms of
$\bE_\mu(\ABS{\nabla f}^2)$ for every smooth function $f$. Note that this task
is impossible if $\mu_0$ and $\mu_1$ have disjoint supports. 

\begin{rem}[Finite mixtures]\label{rm:fimix}
  Let $(\mu_i)_{1\leq i\leq n}$ be a family of probability measures on
  $\dR^d$. Consider the finite mixture $\mu=\cM(\nu,(\mu_i)_{1\leq i\leq n})$
  with mixing measure $\nu=p_1\de_1+\cdots+p_n\de_n$. The decomposition of
  variance is a general fact valid in particular for $\mu$, and writes
  $$
  \var{\mu}{f}=\bE_\nu\PAR{\te\mapsto\var{\mu_\te}{f}}%
  +\var{\nu}{\te\mapsto\bE_{\mu_\te}f}.
  $$
  Here again, the first term in the right hand side may be controlled with the
  Poincar\'e inequality for each of the components $(\mu_i)_{1\leq i\leq n}$.
  For the second term of the right hand side, it remains to notice that for
  every $g:\Te=\{1,\ldots,n\}\to\dR$,
  $$
  \var{\nu}{g} =\frac{1}{2}\sum_{i,j} p_ip_j(g(i)-g(j))^2 %
  =\sum_{i<j} p_ip_j(g(i)-g(j))^2
  $$
  which gives for $g=\bE_{\mu_\te}(f)$ the identity
  $$
  \var{\nu}{\bE_{\mu_\te}f} %
  = \sum_{i<j} p_ip_j\PAR{\bE_{\mu_i}f-\bE_{\mu_j}f}^2.
  $$
  As for the two-component case, this further reduces the Poincar\'e
  inequality for $\mu$ to the control of the mean-differences
  $\PAR{\bE_{\mu_i}f-\bE_{\mu_j}f}^2$ in terms of $\bE_\mu(|\nabla f|^2)$. An
  analogous approach for the entropy and the logarithmic Sobolev inequality
  can be obtained by using \cite[Theorem A1 p. 49]{MR1410112} for instance.
\end{rem}

\subsection{Control of the mean-difference in dimension one}

The following lemma provides the control of the mean-difference term
$\PAR{\bE_{\mu_0}f-\bE_{\mu_1}f}^2$ in the case where $\mu_0$ and $\mu_1$ are
probability measures on $\dR$ (\emph{i.e.} $d=1$).

\begin{lem}[Control of the mean-difference term in dimension one]\label{le:Ip}
  Let $\mu_0$ and $\mu_1$ be two probability distributions on $\dR$ absolutely
  continuous with respect to the Lebesgue measure. Let us denote by $F_0$
  (respectively $F_1$) the cumulative distribution function and $f_0$
  (respectively $f_1$) the probability density function of $\mu_0$
  (respectively $\mu_1$). If $\mathrm{co}(S)$ denotes the convex envelope of
  the set $S=\mathrm{supp}(\mu_0)\cup\mathrm{supp}(\mu_1)$, then, for any
  $p\in(0,1)$, with $\mu_p=p\mu_1+q\mu_0$ and $q=1-p$, we have
  $$
  \PAR{\bE_{\mu_0}f-\bE_{\mu_1}f}^2 %
  \leq I(p)\,\bE_{\mu_p}(f'^2) %
  \quad\text{where}\quad%
  I(p)%
  =\int_{\mathrm{co}(S)}\!\frac{\PAR{F_1(x)-F_0(x)}^2}{pf_1(x)+qf_0(x)}\,dx,
  $$
  and the constant $I(p)$ cannot be improved. Moreover, the function $p\mapsto
  I(p)$ is convex, and
  \begin{equation}
    \label{eq:encadreI}
     \frac{1}{2\max(p,q)}I\PAR{\frac{1}{2}} %
     \leq I(p) %
     \leq \frac{1}{2\min(p,q)}I\PAR{\frac{1}{2}}.
   \end{equation}
   Furthermore, if $S$ is not connected then $I$ is constant and equal to
   $\infty$, while the convexity of $I$ ensure that
   $\sup_{p\in(0,1)}I(p)=\max(I(0^+),I(1^-))$ where
   $$
   I(0^+)=\lim_{p\to 0^+}I(p) %
   \quad\text{and}\quad %
   I(1^-)=\lim_{p\to 1^-}I(p),
   $$  
   and $I(p)<\infty$ for every $p$ in $(0,1)$ if and only if
   $\max(I(0^+),I(1^-))<\infty$.
 \end{lem}

% \begin{rem}[How to bound $I(p)$]
%   The definition of $I(p)$ in Lemma \ref{le:Ip} gives
%   $$
%   I\PAR{p}
%   \leq \PAR{\sup_{x\in\mathrm{co}(S)}\frac{|F_1(x)-F_0(x)|}{pf_1(x)+qf_0(x)}}^2
%   $$
%   for every $p\in(0,1)$. In particular, if the right hand side is finite then
%   $I(p)$ is finite. Lemma \ref{le:Ip} states a stronger criterion of
%   finiteness by using the convexity of $p\mapsto I(p)$.
% \end{rem}

\begin{proof}[Proof of Lemma \ref{le:Ip}]
  For any smooth and compactly supported function $f$, an integration by parts
  gives for every $\te\in\{0,1\}$,
  $$
  \bE_{\mu_\theta}f=\int_\dR\!f(x)f_\theta(x)\,dx
  =-\int_\dR\!f'(x)F_\theta(x)\,dx. 
  $$
  Since $F_1-F_0=0$ outside $\mathrm{co}(S)$ we have
  $$
  \bE_{\mu_0}f-\bE_{\mu_1}f=\int_{\mathrm{co}(S)}\!(F_1(x)-F_0(x))f'(x)\,dx.
  $$
  It remains to use the Cauchy-Schwarz inequality, which gives
  \begin{align*}
    \PAR{\bE_{\mu_0}f-\bE_{\mu_1}f}^2
    &=\PAR{\int_{\mathrm{co}(S)}\!\frac{F_0(x)-F_1(x)}{\sqrt{pf_1(x)+qf_0(x)}}
      f'(x){\sqrt{pf_1(x)+qf_0(x)}}\,dx}^2\\
    &\leq I(p)\,\int_{\mathrm{co}(S)}\! f'(x)^2 (pf_1(x)+qf_0(x))  dx %
    =I(p)\bE_{\mu_p}(f'^2).
  \end{align*}
  The equality case of the Cauchy-Schwarz inequality provides the optimality
  of $I(p)$. The bound \eqref{eq:encadreI} follows from
  $2\min(p,q)(f_0+f_1)/2\leq pf_1+qf_0\leq 2\max(p,q)(f_0+f_1)/2$. The other
  claims of the lemma are immediate.
\end{proof}

\subsection{Control of the Poincar\'e and logarithmic Sobolev constants}

By combining the decomposition of the variance and of the entropy given at the
beginning of the current section with Lemma \ref{le:Ip} and Lemma
\ref{le:gi-bern}, we obtain the following Theorem.

\begin{thm}[Poincar\'e and logarithmic Sobolev inequalities for two-component
  mixtures]
  \label{th:pi-lsi-2-comp}
  Let $\mu_0$ and $\mu_1$ be two probability distributions on $\dR$ absolutely
  continuous with respect to the Lebesgue measure, and consider the
  two-component mixture $\mu_p=p\mu_1+q\mu_0$ with $0\leq p\leq 1$ and
  $q=1-p$. If $I(p)$ is as in Lemma \ref{le:Ip} then for every $p\in(0,1)$,
  $$
  C_\textsc{PI}(\mu_p)%
  \leq \max(C_\textsc{PI}(\mu_0),C_\textsc{PI}(\mu_1))+pq I(p) 
  $$ 
  and 
  $$
  C_\textsc{GI}(\mu_p) %
  \leq\max(C_\textsc{GI}(\mu_0),C_\textsc{GI}(\mu_1))%
  +\frac{\log q-\log p}{q-p} % 
\PAR{pqI(p)+ \max(C_\textsc{PI}(\mu_0),C_\textsc{PI}(\mu_1))}.
  $$
  In particular, since $\sup_{p\in(0,1)}I(p)=\max(I(0^+),I(1^-))$ where
  $I(0^+)$ and $I(1^-)$ are as in Lemma \ref{le:Ip}, we get the following
  uniform bound:
  $$
  \sup_{p\in(0,1)}
  C_\textsc{PI}(\mu_p)%
  \leq \max(C_\textsc{PI}(\mu_0),C_\textsc{PI}(\mu_1))
  +\frac{1}{4}\max(I(0^+),I(1^-)).
  $$ 
  Moreover, if $I(0^+)<\infty$ (respectively if $I(1^-)<\infty$) then
  $$
  \limsup_{p\to0^+\text{respectively $1^-$}} %
  C_\textsc{PI}(\mu_p)\leq\max(C_\textsc{PI}(\mu_0),C_\textsc{PI}(\mu_1)) 
  $$
\end{thm}

The upper bounds given by Theorem \ref{th:pi-lsi-2-comp} must be understood in
$[0,\infty]$ since the right hand side can be infinite (in such a case the
bound is of course useless). Additionally, by Lemma \ref{le:Ip}, the function
$p\mapsto I(p)$ is convex, and it is possible that $I(1/2)<\infty$ while
$\max(I(0^+),I(1^-))=\infty$. The following corollary provides a uniform bound
on the Poincar\'e constant of a two-component mixture in terms of $I(1/2)$
without using $\max(I(0^+),I(1^-))$. This corollary has no immediate
logarithmic Sobolev counterpart, as explained in the remark below following
the proof of the corollary.

\begin{cor}[Uniform Poincar\'e inequality for two-component
  mixtures]\label{co:unif-2-pi}
  Let $\mu_0$ and $\mu_1$ be two probability distributions on $\dR$ absolutely
  continuous with respect to the Lebesgue measure and consider the mixture
  $\mu_p=p\mu_1+q\mu_0$ for every $p\in[0,1]$. We have then
  $$
  \max_{p\in[0,1]}C_\textsc{PI}(\mu_p)
  \leq \max(C_\textsc{PI}(\mu_0),C_\textsc{PI}(\mu_1))
  +\frac{1}{2}I\PAR{\frac{1}{2}}
  $$
  where $I(1/2)$ is as in Lemma \ref{le:Ip}.
\end{cor}

\begin{proof}[Proof of Corollary \ref{co:unif-2-pi}]
  We observe that thanks to \eqref{eq:encadreI}, one has
  $$
  pqI(p)=\max(p,q)\min(p,q)I(p)\leq
  \frac{1}{2}I\PAR{\frac{1}{2}}
  $$
  and Theorem \ref{th:pi-lsi-2-comp} provides the desired result.
\end{proof}

\begin{rem}[Blow-up of the logarithmic Sobolev constant]
  With the notations of Corollary \ref{co:unif-2-pi}, we have, by using the
  same argument, that for every $p\in(0,1)$, 
  $$
  C_\textsc{GI}(\mu_p)
  \leq \max(C_\textsc{GI}(\mu_0),C_\textsc{GI}(\mu_1))
  +\frac{1}{2}\frac{\log(q)-\log(p)}{q-p}
  \PAR{I\PAR{\frac{1}{2}}+\max(C_\textsc{PI}(\mu_0),C_\textsc{PI}(\mu_1))}.
  $$
  Since $(\log(q)-\log(p))/(q-p)$ goes to $+\infty$ at speed
  $-\log(\min(p,q))$ as $\min(p,q)$ goes to $0$, we cannot derive a uniform
  logarithmic Sobolev inequality for two-component mixtures under the sole
  assumption that $I(1/2)<\infty$. Surprisingly, we shall see in the sequel
  that this behavior is sharp and cannot be improved in general for
  two-component mixtures.
\end{rem}

\subsection{The fundamental example of two Gaussians with identical variance}
\label{ss:mixgau}

It was already observed by Johnson in \cite[Theorem 1.1 page 536]{MR2141356}
that for the finite univariate Gaussian mixture
$\mu=p_1\cN(m_1,\tau^2)+\cdots+p_n\cN(m_n,\tau^2)$, we have
$$
C_\textsc{PI}(\mu) %
\leq \tau\PAR{1+\frac{\sigma^2}{\tau\min_{1\leq i\leq n}
    p_i}\exp\PAR{\frac{\sigma^2}{\tau \min_{1\leq i\leq n} p_i}}}
$$
where $\si^2=(p_1m_1^2+\cdots+p_nm_n^2)-(p_1m_1+\cdots+p_nm_n)^2$ is the
variance of $p_1\de_{m_1}+\cdots+p_n\de_{m_n}$. This upper bound on the
Poincar\'e constant blows up as $\min_{1\leq i\leq n}p_i$ goes to $0$. Madras
and Randall have also obtained \cite[Theorem 1.2 and Section 5]{MR1910641}
upper bounds for the Poincar\'e constant of non-Gaussian finite mixtures under
an overlapping condition on the supports of the components. As for the result
of Johnson mentioned earlier, their upper bound blows up when the minimum
weight of the mixing law $\min_{1\leq i\leq n}p_i$ goes to $0$. In the sequel,
we show that the Poincar\'e constant can remain actually bounded as
$\min_{1\leq i\leq n}p_i$ goes to $0$. To fix ideas, we will consider the
special case of a two-component mixture of two Gaussian distributions
$\cN(-a,1)$ and $\cN(+a,1)$. As usual, we denote by $\Phi$ (respectively
$\varphi$) the cumulative distribution function (respectively probability
density function) of the standard Gaussian measure $\cN(0,1)$.

\begin{cor}[Mixture of two Gaussians with identical variance]
  \label{co:twogau}
  For any $a>0$ and $0<p<1$, let $\mu_0=\cN(-a,1)$ and $\mu_1=\cN(+a,1)$, and
  define the two-component mixture $\mu_p=p\mu_1+q\mu_0$. Then
  $$
  C_\textsc{PI}(\mu_p) %
  \leq 1+pq
  4a^2\PAR{\Phi(2a)e^{4a^2}+\frac{2a}{\sqrt{2\pi}}e^{2a^2}+\frac{1}{2}}
  $$
  Additionally, a sharper upper bound for $p=1/2$ is given by
  $$
  C_\textsc{PI}(\mu_{1/2}) \leq 
  1+a\frac{2\Phi(a)-1}{2\vphi(a)}.
  $$
\end{cor}

Note that as a function of $p$, the obtained upper bounds on the constants are
continuous on the whole interval $[0,1]$. The bound \eqref{eq:K-PI-GI}
expressed in the univariate situation implies that $C_\textsc{PI}$ is always
greater than or equal to the variance of the probability measure. Here, the
variance of $\mu_p$ is equal to $1+4apq$. Then the upper bound on the
Poincar\'e constant given above is sharp for any $p\in(0,1)$ as $a$ goes to
$0$.

\begin{proof}[Proof of Corollary \ref{co:twogau}]
  Lemma \ref{le:Ip} ensures that $p\mapsto I(p)$ is a convex function: let us
  have a look at $I(0^+)$ and $I(1^-)$ which are here equal by symmetry. Since
  $$
  \Phi(x+a)-\Phi(x-a)=\int_{-a}^{+a}\!\vphi(x+u)\,du%
  \leq 2a
  \begin{cases}
    \vphi(x+a) &\text{if }x< -a,\\
    \vphi(0) &\text{if }-a\leq x\leq a,\\
    \vphi(x-a) &\text{if }a<x,
  \end{cases}
  $$
  one has 
  \begin{align*}
    I(1^-)&=\int_\dR\!\frac{(\Phi(x+a)-\Phi(x-a))^2}{\vphi(x-a)}\,dx\\
    &\leq 4
    a^2\PAR{\int_{-\infty}^{-a}\frac{\vphi(x+a)^2}{\vphi(x-a)}\,dx +
      \vphi(0)^2\int_{-a}^{+a}\frac{1}{\vphi(x-a)}\,dx
      + \int_{+a}^{+\infty}\!\!\vphi(x-a)\,dx}\\
    &\leq 4
    a^2\PAR{e^{4a^2}\int_{-\infty}^{-a}e^{-\frac{(x+3a)^2}{2}}\frac{1}{\sqrt{2\pi}}\,dx
      + \frac{1}{\sqrt{2\pi}}\int_{0}^{2a}e^{\frac{x^2}{2}}\,dx
      + \int_{0}^{+\infty}\!\!\vphi(x)\,dx}\\
    &\leq
    4a^2\PAR{\Phi(2a)e^{4a^2}+\frac{2a}{\sqrt{2\pi}}e^{2a^2}+\frac{1}{2}}.
  \end{align*}
  Then, the first statement follows from Theorem
  \ref{th:pi-lsi-2-comp}. For the second one, by Lemma
  \ref{le:band-bound} given at the end of the section, we have
  \begin{align*}
    I\PAR{\frac{1}{2}} &=2\int_\dR\!%
    \frac{\Phi(x+a)-\Phi(x-a)}{\vphi(x+a)+\vphi(x-a)}(\Phi(x+a)-\Phi(x-a))\,dx
    \\%
    &\leq 2\tau_a \int_\dR\!(\Phi(x+a)-\Phi(x-a))\,dx \\
    &=4a\tau_a.
  \end{align*}
  This gives as expected $I(1/2)\leq 2a(2\Phi(a)-1)/\varphi(a)$.
\end{proof}

The following lemma shows that $I(1/2)$ is related to some kind of ``band
isoperimetry''. Note that Lemma \ref{le:Ip} provides a more general approach
beyond the Gaussian case.

\begin{lem}[Band bound]\label{le:band-bound}
  For any $x\in\dR$ and any $a>0$,
  $$
  \frac{\Phi(x+a)-\Phi(x-a)}{\vphi(x+a)+\vphi(x-a)} %
  \leq \frac{\Phi(+a)-\Phi(-a)}{\vphi(+a)+\vphi(-a)} %
  =\tau_a
  $$
  Moreover, this constant cannot be improved. As an example, one has
  $\tau_1\approx 1.410686134$.
\end{lem}

\begin{proof}[Proof of Lemma \ref{le:band-bound}]
  Assume that $a=1$. Let $\tau>0$ and define for any $x\in\dR$
  $$
  \al(x)=\Phi(x+1)-\Phi(x-1)-\tau(\vphi(x+1)+\vphi(x-1)).
  $$
  One has $\al'(x)=0$ iff $\tau(1+x+(x-1)e^{2x})=e^{2x}-1$. Thus, either $x=0$,
  or 
  $$
  \tau^{-1} = \be(x)=-1+x\coth(x).
  $$
  The function $\be$ is even, convex, and achieves its global minimum $0$ at
  $x=0$. Therefore, the equation $\al'(x)=0$ has three solutions
  $\{-x_\tau,0,+x_\tau\}$, where $x_\tau>0$ satisfies $\tau\be(x_\tau)=1$. Since
  $\lim_{x\to\pm\infty}\al(x)=0$, one has $\al\leq0$ on $\dR$ if and only if
  $\al(0)\leq0$ and $\al''(0)\leq0$. The condition $\al(0)\leq0$ is fulfilled
  as soon as 
  $$
  \tau\geq \frac{\Phi(+1)-\Phi(-1)}{\vphi(+1)+\vphi(-1)}
  $$
  whereas the condition $\al''(0)\geq0$ holds for any $\tau$. The case where
  $a\neq1$ is similar.
\end{proof}

\begin{rem}[Relation with isoperimetry]
  If $A_x=[x-a,x+a]$ then $\pd A_x=\{x-a,x+a\}$. If $\ga=\cN(0,1)$ then
  $\ga(A_x)=\Phi(x+a)-\Phi(x-a)$ while $\ga_s(\pd A_x)=\vphi(x+a)+\vphi(x-a)$
  where $\gamma_s$ is the surface measure associated to $\gamma$, see e.g.
  \cite{MR1957087}. Lemma \ref{le:band-bound} expresses that for any
  $A\in\cC_a=\{A_x;x\in\dR\}$, we have $\ga(A) \leq \tau_a\ga_s(\pd A)$ and
  equality is achieved for $A=A_0$. Recall that the Gaussian isoperimetric
  inequality states that $(\vphi\circ\Phi^{-1})(\ga(A)) \leq \ga_s(\pd A)$ for
  any regular $A\subset\dR$ with equality when $A$ is a half line, see e.g.
  \cite{MR1957087} and references therein.
\end{rem}

\subsection{Gallery of examples of one-dimensional two-component mixtures}

Recall that if $\mu$ is a probability measure on $\dR$ with density $f>0$ and
median $m$ then
\begin{equation}\label{eq:hardy}
\max(b_-,b_+)\leq C_\textsc{GI}(\mu) \leq 16 \max(b_-,b_+)
\end{equation}
where
\begin{align*}
  b_+&=\sup_{x>m}\mu([x,+\infty))\log\PAR{1+\frac{1}{2\mu([x,+\infty))}}\int_m^x\!\frac{1}{f(y)}\,dy,\\
  % B_+&=\sup_{x>m}\mu([x,+\infty))\log\PAR{1+\frac{e^2}{\mu([x,+\infty))}}\int_m^x\!\frac{1}{f(y)}\,dy
\end{align*}
and
\begin{align*}
  b_-&=\sup_{x<m}\mu((-\infty,x])\log\PAR{1+\frac{1}{2\mu((-\infty,x])}}\int_x^m\!\frac{1}{f(y)}\,dy.\\
  % B_-&=\sup_{x>m}\mu((-\infty,x])\log\PAR{1+\frac{e^2}{\mu((-\infty,x])}}\int_x^m\!\frac{1}{f(y)}\,dy.
\end{align*}
These bounds appear in \cite[Remark 7 page 9]{br} as a refinement of a famous
criterion by Bobkov and G\"otze based on previous works of Hardy and
Muckenhoupt, see also \cite{miclo}. More generally, the notion of measure
capacities constitutes a powerful tool for the control of $C_\textsc{PI}$ and
$C_\textsc{GI}$, see \cite{maz} and \cite{bcr,br}. In the present article, we
only use a weak version of such criteria, stated in the following lemma, and
which can be found for instance in \cite[Chapter 6 page 107]{MR2002g:46132}.
We will typically use it in order to show that $C_\textsc{GI}(p_1\mu+q\mu_0)$
blows up as $p$ goes to $0$ or $1$ for certain choices of $\mu_0$ and $\mu_1$.

\begin{lem}[Crude lower bound]\label{le:crude}
  Let $\mu$ be some distribution on $\dR$ with density $f>0$ then for every
  median $m$ of $\mu$ and every $x\leq m$, by denoting $\Psi(u)=-u\log(u)$,
  $$
  150\,C_\textsc{GI}(\mu)\geq \Psi(\mu(-\infty,x])\int_x^m\!\frac{1}{f(y)}\,dy.
  $$
\end{lem}

In this whole section, $\mu_0$ and $\mu_1$ are absolutely continuous
probability measures on $\dR$ with cumulative distribution functions $F_0$ and
$F_1$ and probability density functions $f_0$ and $f_1$. For every $0\leq
p\leq 1$, we consider the two-component mixture $\mu_p=p\mu_1+q\mu_0$. The
sharp analysis of the logarithmic Sobolev constant for finite mixtures is a
difficult problem. Also, we decided to focus on some enlightening examples, by
providing a gallery of special cases of $\mu_0$ and $\mu_1$ for which we are
able to control the dependence over $p$ of the Poincar\'e and logarithmic
Sobolev constant of $\mu_p$. Some of them are quite surprising and reveal
hidden subtleties of the logarithmic Sobolev inequality as $\min(p,q)$ goes to
$0$\ldots The key tools used here are Theorem \ref{th:pi-lsi-2-comp} and Lemma
\ref{le:crude}.

\subsubsection{One Gaussian and a sub-Gaussian}
\label{sss:gau-sub-gau}

\noindent\textbf{Setting.}
Here $\mu_1=\cN(0,1)$ while $\mu_0$ is such that $f_0\leq \kappa f_1$ for some
finite constant $\kappa\geq1$.

\noindent\textbf{Claim.} For every $0<p<1$ we have
$C_\textsc{PI}(\mu_p)\leq\max(1,C_\textsc{PI}(\mu_0))+Dq$. This upper bound
goes to $\max(1,C_\textsc{PI}(\mu_0))$ as $p\to1$ and is additionally
uniformly bounded when $p$ runs over $(0,1)$. Similarly,
$C_\textsc{GI}(\mu_p)\leq\al-\be\log(p)$ for some constants $\al,\be>0$ which
do not depend on $p$. This upper bound blows up at speed $-\log(p)$ as
$p\to0$. This is actually the real behavior of $C_\textsc{GI}(\mu_p)$ in some
situations as shown in Section \ref{sss:gau-uni}! 

\noindent\textbf{Proof.} 
Since $\mu_1=\cN(0,1)$, we have $C_\textsc{PI}(\mu_1)=1$ and
$C_\textsc{GI}(\mu_1)=2$. By hypothesis, we have $F_0\leq \kappa F_1$ and
$1-F_0 \leq \kappa(1-F_1)$. Thus, for some $D>0$ and every $0<p<1$,
$$
I(p) \leq 
\frac{2(1+\kappa^2)}{p} %
\PAR{\int_{-\infty}^0\!\frac{F_1^2(x)}{f_1(x)}\,dx%
  +\int_0^{+\infty}\!\frac{(1-F_1(x))^2}{f_1(x)}\,dx}=\frac{D}{p}<\infty.
$$
Now Theorem \ref{th:pi-lsi-2-comp} shows that
$C_\textsc{PI}(\mu_p)\leq\max(1,C_\textsc{PI}(\mu_0))+Dq$. The desired upper
bound for $C_\textsc{GI}(\mu_p)$ follows by the same way and we leave the
details to the reader.

\subsubsection{Two Gaussians with identical mean}
\label{sss:gau-gau}

We have already considered the mixture of two Gaussians with identical
variances and different means in Section \ref{ss:mixgau}. Here we consider a
mixture of two Gaussians with identical means and different variances. It
turns out that this Gaussian mixture is a simple Gaussian sub-case of Section
\ref{sss:gau-sub-gau}, for which we are able to provide a more precise bound
for $C_\textsc{GI}$.

\noindent\textbf{Setting.} $\mu_1=\cN(0,\sigma^2)$ with $\sigma>1$ and
$\mu_0=\cN(0,1)$.

\noindent\textbf{Claim.} 
There exists $C>0$ such that, for any $p<1/2$, 
$$
I(p)\leq C\PAR{\frac{1}{p}}^{\frac{\sigma^2-2}{\sigma^2-1}}%
\quad\text{and}\quad%
C_{\textsc{PI}}(\mu_p)\leq \sigma^2+Cp^{\frac{1}{\sigma^2-1}}.
$$
Moreover we have $\sup_{p\in(0,1)} C_\textsc{PI}(\mu_p)<\infty$. 

\noindent\textbf{Proof.}
We have $f_0\leq \kappa f_1$ for some $\kappa>1$, and we recover the setting
of Section \ref{sss:gau-sub-gau}. Let us provide now an upper bound for $I(p)$
when $p$ is close to $0$. We have $pf_1(x)\geq qf_0(x)$ if and only if
$\ABS{x}\geq \OL{x}_p$ where
$$
\OL{x}_p=\sqrt{\frac{2\sigma^2}{\sigma^2-1}\log\PAR{\frac{q\sigma}{p}}}.
$$
We have, for some constant $C>0$,
\begin{align*}
I(p)&\leq 2\int_{-\infty}^{-1}\frac{F_1(x)^2}{pf_1(x)+qf_0(x)}\,dx%
+2\int_{-1}^0\frac{F_1(x)^2}{f_0(x)}dx\\%
&\leq 2\int_{-\infty}^{-1}\frac{1}{x^2}\frac{f_1(x)^2}{pf_1(x)+qf_0(x)}\,dx+C,
\end{align*}
since $2q\geq 1$ and $F_1(x)\leq f_1(x)/\ABS{x}$. If $p$ is sufficiently
small then $\OL{x}_p>1$ and
$$
\int_{-\infty}^{-1}\frac{1}{x^2}\frac{f_1(x)^2}{pf_1(x)+qf_0(x)}\,dx%
\leq 2\int_{-\OL{x}_p}^{-1}\frac{1}{x^2}\frac{f_1(x)^2}{f_0(x)}\,dx%
+\frac{1}{p}F_1(-\OL{x}_p).
$$
By the definition of $\OL{x}_p$, for some $C>0$,
$$
\frac{1}{p}F_1(-\OL{x}_p)\leq \frac{C}{p} e^{-\OL{x}_p^2/(2\sigma^2)}%
\leq C \PAR{\frac{1}{p}}^{\frac{\sigma^2-2}{\sigma^2-1}}.
$$
If $\sigma^2\leq 2$, then this function of $p$ is bounded. On the other hand,
for some $C>0$,
$$
\int_{-\OL{x}_p}^{-1}\frac{1}{x^2}\frac{f_1(x)^2}{f_0(x)}\,dx%
\leq%
C\int_{-\OL{x}_p}^{-1} \frac{1}{x^2}e^{\frac{\sigma^2-2}{2\sigma^2}x^2}\,dx.
$$
If $\sigma^2\leq 2$, then this function of $p$ is bounded. If $\sigma^2>2$,
then, for some $C>0$,
$$
\int_{-\OL{x}_p}^{-1} e^{\frac{\sigma^2-2}{2\sigma^2}x^2}\,dx%
\leq C e^{\frac{\sigma^2-2}{2\sigma^2}\OL{x}_p^2}%
\leq C \PAR{\frac{1}{p}}^{\frac{\sigma^2-2}{\sigma^2-1}}.
$$
As a conclusion, if $\sigma^2\leq 2$, then $\sup_{p\in(0,1)}I(p)<\infty$,
whereas if $\sigma^2> 2$, then for some constant $C>0$ and any $p<1/2$,
$$
I(p)\leq C \PAR{\frac{1}{p}}^{\frac{\sigma^2-2}{\sigma^2-1}}
$$
The bound of $C_{\textsc{PI}}$ follows from Theorem
\ref{th:pi-lsi-2-comp}. For the logarithmic Sobolev inequality, one may use 
the Bobkov-G\"otze criterion.

\subsubsection{Two uniforms with overlapping supports}
\label{sss:uni-uni}

\noindent\textbf{Setting.} 
Here $\mu_0=\cU([0,1])$ and $\mu_1=\cU([a,a+1])$ for some $a\in [0,1]$.

\noindent\textbf{Claim.} For every $p\in(0,1)$, we have
$$
C_\textsc{PI}(\mu_p)\leq \pi^{-2}+\frac{a^2}{3}\PAR{3pq(1-a)+a} %
$$
and for $p\leq 1/2$,
$$
C_\textsc{GI}(\mu_p)\geq
\frac{a^2}{600}\log(1/p).
%2\pi^{-2}+\frac{\log(q)-\log(p)}{q-p}\frac{a^2}{3}\PAR{3pq(1-a)+a}.
$$
\noindent\textbf{Proof.}
It is known (see e.g. \cite{gentil}) that $C_\textsc{PI}(\cU([0,1])=\pi^{-2}$
while $C_\textsc{GI}(\cU([0,1])=2\pi^{-2}$. By translation invariance, we also
have $C_\textsc{PI}(\cU([1,1+a])=\pi^{-2}$ and
$C_\textsc{GI}(\cU([1,1+a])=2\pi^{-2}$. The desired result follows from
Theorem \ref{th:pi-lsi-2-comp} since for $p\in(0,1)$,
$$
I(p)%
=\int_0^a\frac{x^2}{q}\,dx+\int_a^1\frac{a^2}{p+q}\,dx%
+\int_1^{a+1}\frac{(1+a-x)^2}{p}\,dx%
=\frac{a^2}{3pq}\PAR{3pq(1-a)+a}.
$$
The minoration of $C_\textsc{GI}(\mu_p)$ follows from Lemma
\ref{le:crude}:
$$
150C_\textsc{GI}(\mu_p)\geq%
\Psi(\mu_p(0,a/2])\int_{a/2}^a\!\frac{1}{f_p(y)}\,dy%
=\Psi\PAR{\frac{pa}{2}}\frac{a}{2p}.
$$

\subsubsection{One Gaussian and a uniform}
\label{sss:gau-uni}

\noindent\textbf{Setting.} 
Here $\mu_1=\cN(0,1)$ and $\mu_0=\cU([-1,+1])$.

\noindent\textbf{Claim.} There exists a real constant $C>0$ such that
$C_\textsc{GI}(\mu_p)\geq -C\log(p)$ for every $p\in(0,1)$. Also,
$C_\textsc{GI}(\mu_p)$ blows up at speed $-\log(p)$ as $p\to0^+$. Moreover,
$\mu_p$ satisfies a sub-Gaussian concentration of measure for Lipschitz
functions, uniformly in $p\in(0,1)$. This similarity with the Bernoulli law
$\cB(p)$ suggests that the blow up phenomenon of $C_\textsc{GI}(\mu_p)$ is due
to the asymptotic support reduction from $\dR$ to $[-1,+1]$ when $p$ goes to
$0^+$. Actually, Section \ref{sss:surprise} shows that this intuition is
false.

\noindent\textbf{Proof.} 
We have $f_0\leq \kappa f_1$ for some constant $\kappa\geq1$. Also, for every
$p\in(0,1)$, the result of Section \ref{sss:gau-sub-gau} gives that
$C_\textsc{GI}(\mu_p)\leq \al-\beta\log(p)$ for some constants $\al>0$ and
$\be>0$ independent of $p$. Now, by Lemma \ref{le:crude},
\begin{align*}
  150\,C_\textsc{GI}(p) %
  &\geq \Psi(pF_1(-2)+qF_0(-2))%
  \int_{-2}^0\!\frac{1}{pf_1(u)+qf_0(u)}\,du \\
  &= \Psi(pF_1(-2))%
  \int_{-2}^0\!\frac{1}{pf_1(u)+qf_0(u)}\,du \\
  &\geq -\PAR{F_1(-2)\int_{-2}^{-1}\!\frac{1}{f_1(u)}\,du}\log(p).
\end{align*}

\subsubsection{Surprising blow up}
\label{sss:surprise}

\noindent\textbf{Setting.} Here $f_1(x)=Z_1^{-1}e^{-x^2}$ and
$f_0(x)=Z_0^{-1}e^{-|x|^a}$ for some fixed real number $a>2$, with
$Z_1=\pi^{-1/2}$ and $Z_0=2\Gamma(a^{-1})a^{-1}$. Note that $\mu_0$ has
lighter tails than $\mu_p$ with $p>0$.

\noindent\textbf{Claim.}
There exists a real constant $C>0$ which may depend on $a$ such that
$$
C_\textsc{GI}(\mu_p)\geq C(-\log(p))^{1-2a^{-1}}
$$
for small enough $p$. In particular, $C_\textsc{GI}(\mu_p)$ blows up as
$p\to0^+$.

\noindent\textbf{Comments.}
As mentioned in the introduction, we have
$\max(C_\textsc{GI}(\mu_0),C_\textsc{GI}(\mu_1))<\infty$. We have seen in
Section \ref{sss:gau-gau} that $C_\textsc{GI}(\mu_p)$ does not blow up as
$p\to0^+$ if $a=2$. Here $a>2$, and $\mu_0$ has strictly lighter tails than
$\mu_p$ for every $p\in(0,1)$, and moreover, this difference is at the level
of the log-power of the tails, not only at the level of the constants in front
of the log-power. The potential (-$\log$-density) of $\mu_p$ has multiple
wells, see Figure \ref{fi:explo}. This example shows also that the blow up
speed of $C_\textsc{GI}(\mu_p)$ as $p\to0^+$ cannot be improved by considering
a mixture of fully supported laws. Note that $\mu_0\to\cU([-1,+1])$ as
$a\to\infty$, and the result is thus compatible with Section
\ref{sss:gau-uni}.

\noindent\textbf{Proof.} 
Since $f_0\leq \kappa f_1$ for some constant $\kappa\geq1$, Section
\ref{sss:gau-sub-gau} gives $C_\textsc{GI}(\mu_p)<\infty$ for every
$p\in(0,1)$. Moreover, $p\mapsto C_\textsc{GI}(\mu_p)$ is uniformly bounded on
$(p_0,1)$ for every $p_0>0$. Let us study the behavior of this function as
$p\to0$. In the sequel we assume that $p<p_0$ where $p_0$ satisfies
$p_0Z_0=q_0Z_1$. The immediate tails comparison gives $qf_0(x)\leq pf_1(x)$
for large enough $x$. Let us find some explicit bound on $x$. The inequality
$qf_0(x)\leq pf_1(x)$ writes $|x|^a-x^2\geq \log(qZ_1)-\log(pZ_0)$. Now,
$|x|^a-x^2\geq \frac{1}{2}|x|^a$ for $|x|^{a-2}\geq 2$. The non-negative
solution of $\frac{1}{2}|x|^a=\log(qZ_1)-\log(pZ_0)$ is
$$
\OL{x}_p=\PAR{2\log\PAR{\frac{q}{p}\frac{Z_1}{Z_0}}}^{1/a}.
$$
If $p$ is small enough, then $|\OL{x}_p|^{a-2}\geq 2$ and therefore,
$qf_0(x)\leq pf_1(x)$ for any $|x|\geq \OL{x}_p$. Now, by Lemma
\ref{le:crude}, for small enough $p$,
$$
150\,C_\textsc{GI}(\mu_p) 
\geq
\Psi(pF_1(-2\OL{x}_p)+qF_0(-2\OL{x}_p))
\int_{-2\OL{x}_p}^0\!\frac{1}{pf_1(u)+qf_0(u)}\,du.
$$
For small enough $p$, we have $\max(F_0,F_1)(-2\OL{x}_p)<e^{-1}$ and thus, for
some constant $C>0$,
$$
\Psi(pF_1(-2\OL{x}_p)+qF_0(-2\OL{x}_p)) %
\geq \Psi(pF_1(-2\OL{x}_p)) %
\geq -pF_1(-2\OL{x}_p)\log(p) %
\geq C\frac{e^{-4\OL{x}_p^2}}{\OL{x}_p}\Psi(p).
$$
On the other hand, since $qf_0(x)\leq pf_1(x)$ for $|x|\geq \OL{x}_p$, we get
for some constant $C>0$,
$$
\int_{-2\OL{x}_p}^0\!\frac{1}{pf_1(u)+qf_0(u)}\,du %
\geq \int_{-2\OL{x}_p}^{-\OL{x}_p}\!\frac{du}{2pf_1(u)} %
\geq \frac{Ce^{4\OL{x}_p^2}}{p\OL{x}_p}.
$$
Consequently, for some real constant $C>0$,
$$
150\,C_\textsc{GI}(\mu_p) \geq -C\frac{\log(p)}{\OL{x}_p^2}.
$$
Now, by using the explicit expression of $\OL{x}_p$, we finally obtain for some
real constant $C>0$,
$$
C_\textsc{GI}(\mu_p) \geq C\,(-\log(p))^{1-2a^{-1}}.
$$

\begin{figure}[htbp]
  \begin{center}
    \includegraphics[scale=0.4]{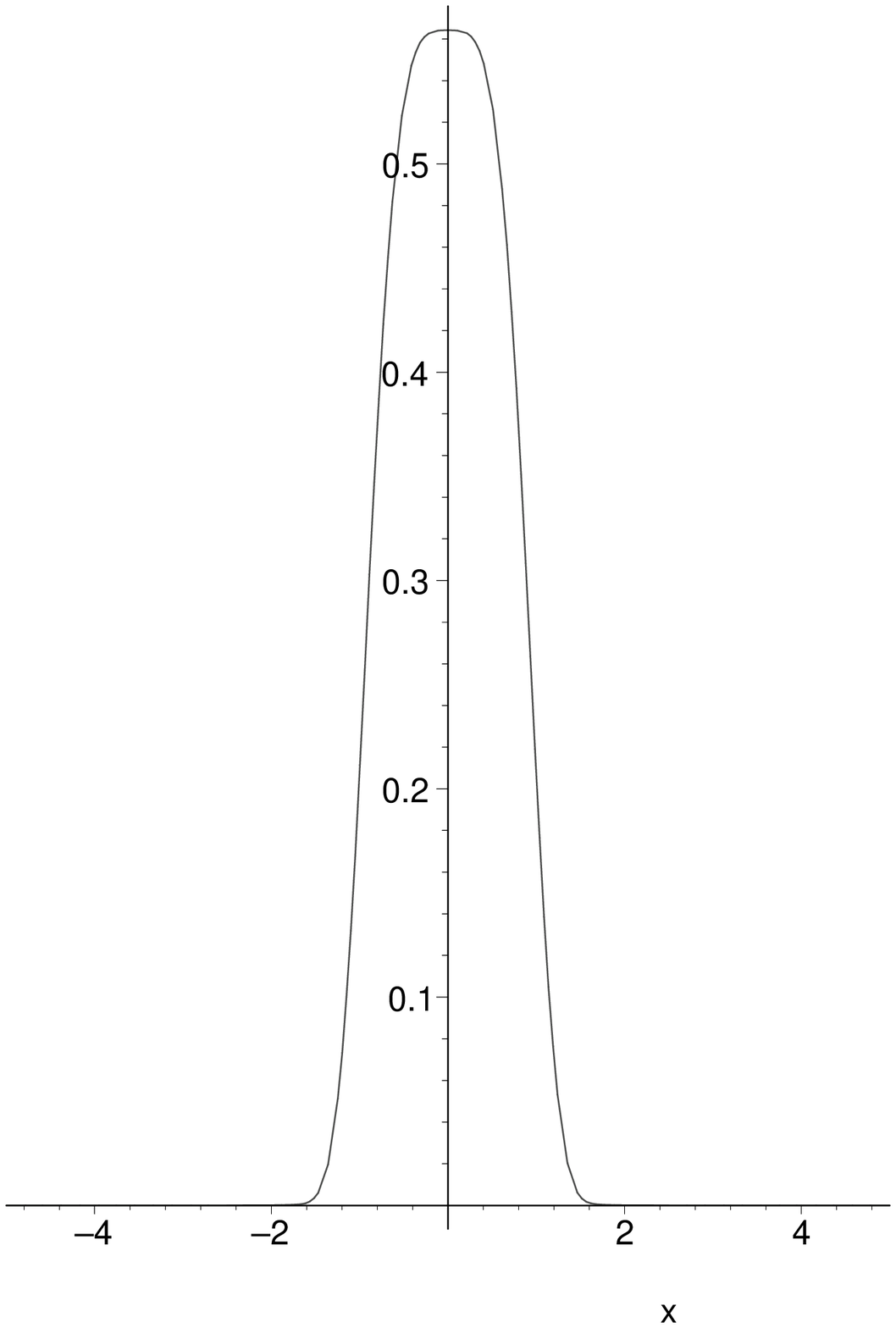}
    \includegraphics[scale=0.4]{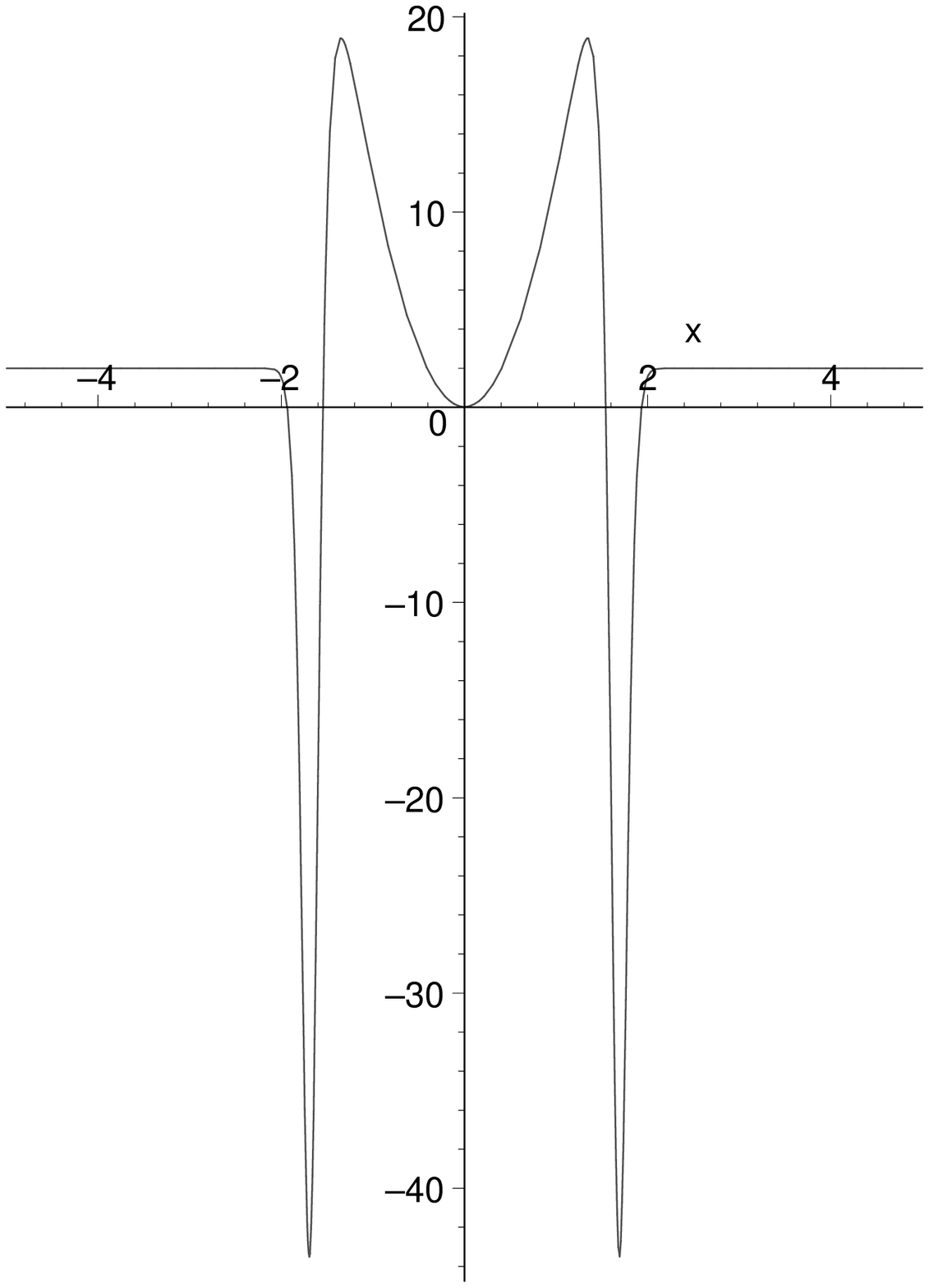}
    \caption{Density and second derivative of $-\log$-density of $\mu_p$ for
      Example \ref{sss:surprise} with $p=1/100$ and $a=4$. The second plot
      reveals a deep multiple wells potential.}
    \label{fi:explo}
  \end{center}
\end{figure}

\subsection{Multivariate mean-difference bound}

It is quite natural to ask for a multidimensional counterpart of the
mean-difference Lemma \ref{le:Ip}. Let us give some informal ideas to attack
this problem. Let $\mu_0$ and $\mu_1$ be two probability measures on $\dR^d$,
and consider as usual the mixture $\mu_p=p\mu_1+q\mu_0$ with $p\in(0,1)$ and
$q=1-p$. It is well known (see for instance \cite{villani-livre}) that if
$\mu_0$ and $\mu_1$ are regular enough, then there exists a map
$T:\dR^d\to\dR^d$ such that the image measure $T\cdot\mu_0$ of $\mu_0$ by $T$
is $\mu_1$ and
$$
W_2(\mu_0,\mu_1)^2=\int_{\dR^d}\!\ABS{T(x)-x}^2\,\mu_0(dx).
$$  
If $\mu_{(s)}$ denotes the image of $\mu_0$ by $x\mapsto sT(x)+(1-s)x$ for
every $0<s<1$, then
$$
\PAR{\bE_{\mu_1}f-\bE_{\mu_0}f}^2
=\PAR{\int_0^1\!\int_{\dR^d}\!(T(x)-x)\cdot\na
  f(sT(x)+(1-s)x)\,d\mu_0(x)\,ds}^2.
$$
By Cauchy-Schwarz's inequality, we get
$$
\PAR{\bE_{\mu_1}f-\bE_{\mu_0}f}^2
\leq \PAR{\int_{\dR^d}\!\ABS{T(x)-x}^2\,d\mu_0(x)}
\PAR{\int_0^1\!\int_{\dR^d}\!\ABS{\na f(x)}^2\,d\mu_{(s)}(x)\,ds}
$$
and therefore %[\cite[Lemma A1]{lott-villani}]
$$
\PAR{\bE_{\mu_1}f-\bE_{\mu_0}f}^2 %
\leq W_2(\mu_1,\mu_0)^2 %
\int_{\dR^d}\!\int_0^1\!\ABS{\na f(x)}^2\,d\mu_{(s)}(x)\,ds.
$$
This shows that in order to control the mean-difference term
$\PAR{\bE_{\mu_1}f-\bE_{\mu_0}f}^2$ by $\bE_{\mu_p}(|\nabla f|^2)$, it is
enough to find a real constant $C_p>0$ such that $\overline{\mu}\leq C_p\mu_p$
where
$$
\overline{\mu}(A)=\int_0^1\!\mu_{(s)}(A)\,ds.
$$
Unfortunately, this is not feasible if for some $s\in(0,1)$, the support of
$\mu_{(s)}$ is not included in the support of $\mu_p$ (union of the supports
of $\mu_0$ and $\mu_1$ if $p\in(0,1)$). This problem is due to the linear
interpolation used to define $\mu_{(s)}$ via $T$. The linear interpolation
will fail if the support of $\mu_p$ is a non-convex connected set. Let us
adopt an alternative pathwise interpolation scheme. For each $x\in
S_0=\mathrm{supp}(\mu_0)$, let us pick a continuous and piecewise smooth
interpolating path $\gamma_x:[0,1]\to\dR^d$ such that $\gamma_x(0)=x$ and
$\gamma_x(1)=T(x)$. Then for every smooth $f:\dR^d\to\dR$,
$$
f(x)-f(T(x))=\int_0^1\!\dot{\gamma}_x(s)\,\nabla f(\gamma_x(s))\,ds
\leq \sqrt{\int_0^1\!|\dot{\gamma}_x(s)|^2\,ds}\ %
\sqrt{\int_0^1\!|\nabla f|^2(\gamma_x(s))\,ds}.
$$
As a consequence, we have 
$$
\PAR{\bE_{\mu_0}f-\bE_{\mu_1}f}^2
\leq 
\PAR{\int_{S_0}\!\int_0^1\!|\dot{\gamma}_x(s)|^2\,ds\,\mu_0(dx)}\ %
\PAR{\int_{S_0}\!\int_0^1\!|\nabla f|^2(\gamma_x(s))\,ds\,\mu_0(dx)}.
$$
Now, let $\mu_{(s)}$ be the image measure of $\mu_0$ by the map
$x\mapsto\gamma_x(s)$, where here again $\overline\mu$ is the measure defined
by $\overline\mu(A)=\int_0^1\!\mu_{(s)}(A)\,ds$. With this notation, we have
$$
\PAR{\bE_{\mu_0}f-\bE_{\mu_1}f}^2
\leq 
\PAR{\int_{S_0}\!\int_0^1\!|\dot{\gamma}_x(s)|^2\,ds\,\mu_0(dx)}\ %
\PAR{\int_{\dR^d}\!|\nabla f|^2(x)\,\overline\mu(dx)}.
$$
Note that
$$
\PAR{\int_{S_0}\!\int_0^1\!|\dot{\gamma}_x(s)|^2\,ds\,\mu_0(dx)}
\geq W_2(\mu_0,\mu_1)^2
$$
with equality when $\gamma_x$ is the linear interpolation map between $x$ and
$T(x)$ for every $x\in S_0$. The mean-difference control that we seek for
follows then immediately if there exists a real constant $C_p>0$ such that
$\overline\mu \leq C_p\mu_p$. The problem is thus reduced to the choice of an
interpolation scheme $\ga$ such that the support of $\overline\mu$ is included
in the support of $\mu_p$ (which is the union of the supports of $\mu_0$ and
$\mu_1$ as soon as $0<p<1$). Let us give now two enlightening examples.

\begin{xpl}[When the linear interpolation map is optimal]\label{ex:linopt}
  Consider the two-\-dimen\-sional example where $\mu_0=\cU([0,2]\times[0,2])$
  and $\mu_1=\cU([1,3]\times[0,2])$. If $\gamma$ is the natural linear
  interpolation map given by $\gamma_x(s)=x+se_1$ then
  $\mu_{(s)}=\cU([s,s+2]\times[0,2])$ is supported inside
  $\mathrm{supp}(\mu_0)\cup\mathrm{supp}(\mu_1)$. This is due to the convexity
  of this union. Also, the linear interpolation map is here optimal. Moreover,
  elementary computations reveal that 
  $$
  C_p=\frac{1}{\min(p,q)}
  \quad\text{and}\quad
  W_2(\mu_0,\mu_1)^2=1.
  $$
  Therefore, for every $0<p<1$ and any smooth $f:\dR^2\to\dR$,
  $$
  \PAR{\bE_{\mu_0}f-\bE_{\mu_1}f}^2 %
  \leq \frac{1}{\min(p,q)}\bE_{\mu_p}(|\nabla f|^2).
  $$
\end{xpl}

\begin{xpl}[When the linear interpolation map fails]\label{ex:linnopt}
  In contrast, for the example where $\mu_0=\cU([0,2]\times[0,2])$ and
  $\mu_1=\cU([1,3]\times[1,3])$ and if $\gamma$ is the natural linear
  interpolation map given by $\gamma_x(s)=x+s(e_1+e_2)$ then $\mu_{(s)}$ is
  not supported in $\mathrm{supp}(\mu_0)\cup\mathrm{supp}(\mu_1)$ and this
  union is not convex. If $A=[0,1]\times[2,3]$ then $\mu_{(s)}(A)>0$ for every
  $0<s<1$ while $\mu_p(A)=0$ for every $0<p<1$ and hence there is no finite
  constant $C_p>0$ such that $\overline\mu\leq C_p\mu_p$. This shows that the
  linear interpolation map fails here. Let us give an alternative
  interpolation map which leads to the desired result. We set for every
  $x\in\mathrm{supp}(\mu_0)$ and every $0\leq s\leq 1$, with
  $\mathbf{1}=(e_1,e_1)$,
  $$
  \gamma_x(s)=
  \begin{cases}
    (1-s)x+2s\mathbf{1} & \text{if $0\leq s\leq \frac{1}{2}$} \\
    sx+\mathbf{1} & \text{otherwise}.
  \end{cases}
  $$
  This corresponds to a two-steps linear interpolation between the squares
  $[0,2]^2$ and $[1,3]^2$ with intermediate square $[1,2]^2$.
  For every $0\leq s\leq 1$, 
  $$
  \mu_{(s)}=
  \begin{cases}
    \cU([2s,2]^2) & \text{if $0\leq s\leq \frac{1}{2}$} \\
    \cU([1,1+2s]^2) & \text{otherwise}.
  \end{cases}
  $$
  Note that we constructed $\gamma$ in such a way that $\mu_{(s)}$ is always
  supported in $\mathrm{supp}(\mu_0)\cup\mathrm{supp}(\mu_1)$. Elementary
  computations reveal that for every $0<p<1$,
  $$
  \int_{S_0}\!\int_0^1\!|\dot{\gamma}_x(s)|^2\,ds\,\mu_0(dx)=\frac{8}{3} %
  \quad\text{and}\quad %
  \overline\mu\leq \frac{4}{\min(p,q)}\mu_p.
  $$
  Finally, putting all together, we obtain for every $0<p<1$ and smooth
  $f:\dR^2\to\dR$,
  $$
  \PAR{\bE_{\mu_0}f-\bE_{\mu_1}f}^2 %
  \leq \frac{32}{3\min(p,q)}\bE_{\mu_p}(|\nabla f|^2).
  $$
  As a conclusion, one can retain that the natural interpolation problem
  associated to the control of the mean-difference involves a kind of
  support-constrained interpolation for mass transportation.
\end{xpl}

\bigskip

{\noindent\textbf{Acknowledgements.} The authors would like to warmly thank
  Arnaud \textsc{Guillin}, S\'ebastien \textsc{Gou\"ezel}, and all the
  participants of the Working Seminar \emph{In\'egalit\'es fonctionnelles}
  held on March 2008 at the Institut Henri Poincar\'e of Paris for fruitful
  discussions. The final form of the paper has benefited from the fine
  comments of two anonymous reviewers.}

\addcontentsline{toc}{section}{\refname}%
{
\footnotesize
%\bibliography{\jobname}
\bibliography{mix}
}
\bibliographystyle{amsplain}

\vfill

\begin{flushright}{\footnotesize\texttt{Compiled \today}}\end{flushright}

\vfill

{\footnotesize %
  \noindent Djalil \textsc{Chafa\"\i} %
  \noindent\url{mailto:chafai(AT)math.univ-toulouse.fr}
  
  \medskip
  \noindent\textsc{UMR181 INRA ENVT, 23 Chemin des Capelles
    F-31076  Toulouse Cedex 3, France.}

  \medskip

  \noindent
  \textsc{UMR5583 CNRS
    Institut de Math\'ematiques de Toulouse (IMT) \\
    Universit\'e de Toulouse III, 118 route de Narbonne, F-31062, Toulouse
    Cedex, France.}

  \bigskip
  
  \noindent Florent \textsc{Malrieu} %
  \url{mailto:florent.malrieu(AT)univ-rennes1.fr}

  \medskip

  \noindent\textsc{UMR 6625 CNRS Institut de Recherche Math\'ematique de
    Rennes (IRMAR) \\ Universit\'e de Rennes I, Campus de Beaulieu, F-35042
    Rennes \textsc{Cedex}, France.}
  
\vfill

}%footnotesize

\end{document}